\documentclass[reqno,12pt]{amsart}
\title{Simplicity criteria for \nagyrev{\'etale} groupoid $C^*$-algebras}
\author{Danny Crytser}
\address{Department of Mathematics, St. Lawrence University, 23 Romoda Drive,
Canton, NY 13617, USA}
\email{danny.crytser@gmail.com}
\author{Gabriel Nagy}
\address{Department of Mathematics, Kansas State University, 1228 N. 17th Street, Manhattan, KS 66506, USA}
\email{nagy@math.ksu.edu}
\usepackage{amssymb,amsmath,amsthm,verbatim,tikz}
\usepackage{graphicx,color}
\usepackage{mathrsfs}


\usepackage{tcolorbox}
\usepackage[margin=1in]{geometry} 
\usepackage[all]{xy}
\usetikzlibrary{decorations.markings}
\usetikzlibrary{arrows}

\def\scol#1{\ifx{#1}{000}
a
\else
; #1
\fi}

\newcounter{myfactnr}
\newcounter{appresult}
\numberwithin{appresult}{section}
\theoremstyle{plain}
\newtheorem{theorem}{Theorem}[section]
\newtheorem{proposition}[theorem]{Proposition}
\newtheorem{lemma}[theorem]{Lemma}
\newtheorem{applemma}[appresult]{Lemma}
\newtheorem{apptheorem}[appresult]{Theorem}
\newtheorem{corollary}[theorem]{Corollary}

\newtheorem{myfact}[myfactnr]{\sc Fact}
\newtheorem*{claim}{\sc Claim}

\newtheorem*{convention}{Convention}

\theoremstyle{definition}
\newtheorem{definition}[theorem]{Definition}
\newtheorem{definitions}[theorem]{Definitions}

\newtheorem{example}[theorem]{Example}
\theoremstyle{remark}
\newtheorem{remark}[theorem]{\sc Remark}

\newtheorem*{note}{Note}
\newtheorem*{notation}{Notation}
\newtheorem*{notations}{Notations}
\newtheorem*{mycomment}{Comment}

\newcounter{oldsec}

\newcommand\nagyrev[1]{\textcolor{black}{#1}}
\newenvironment{nagyrevenv}{\color{black}}{\color{black}}

\newcommand\nagysecondrev[1]{\textcolor{black}{#1}}

\usepackage{hyperref}

\begin{document}

\begin{abstract}
We develop a framework suitable for obtaining simplicity criteria for 
 reduced $C^*$-algebras of Hausdorff \'etale groupoids. This is based on the study of certain
non-degenerate $C^*$-subalgebras (in the case of groupoids, the $C^*$-algebra of the interior isotropy bundle), 
for which one can control (non-unique) state extensions to the ambient $C^*$-algebra. As an application, we give simplicity criteria for reduced crossed products $C_0(Q)\rtimes_\text{red} G$ by discrete groups. 
\end{abstract}

\maketitle

\section*{Introduction}

The task of determining which $C^*$-algebras are simple (in the sense of having no proper, non-trivial norm-closed two-sided ideals) has interested operator algebraists for decades. We follow the approach that has focused on classes of $C^*$-algebras defined from topological, combinatorial, or dynamic information. \nagyrev{There} is a simplicity criterion for graph $C^*$-algebras (\cite[Cor. 3.11]{KPR})\nagyrev{; this criterion was generalized to $k$-graphs} (\cite[Prop. 4.8]{KP}); there is a simplicity criterion for full \'etale groupoid $C^*$-algebras (\cite[Thm 5.1]{BCFS}); and there is a well-known sufficient condition for a group crossed product $C^*$-algebra to be simple 
(\cite{AS}). 
In each of these situations, there is a certain $C^*$-subalgebra $B \subset A$, 
such that simplicity of $A$ is equivalent to the non-existence of non-trivial proper ideals of $B$ which are \emph{invariant} in some suitable fashion. 
In a recent series of papers (\cite{NR}, \cite{NR2}, \cite{BNR}, \cite{BNRSW}) there has been a concerted push to understand the {\em uniqueness\/} property for $C^*$-inclusions for graph, $k$-graph, and groupoid $C^*$-algebras. The perspective that has emerged from these investigations shows that the question of uniqueness is closely tied up with the existence of sufficiently many states on $C^*$-subalgebras that extend uniquely to the ambient $C^*$-algebra.
 In particular, in (\cite[Thm. 3.1]{BNRSW}) it is proved that, for an \'etale groupoid with mild additional requirements, a representation of $C^*_\text{red}(\mathcal{G})$ is faithful if and only if its 
restriction to $C^*_\text{red}(\operatorname{IntIso}(\mathcal{G}))$ is faithful.

In this paper we propose an approach that aims to unify all these results, by identifying such essential $C^*$-inclusions in a manner that circumvents the unique state extension issues that were previously employed.
In particular, in Section \ref{rel-dom-sec}, we extend the abstract uniqueness theorem of \cite{BNRSW} by developing the  notions of \emph{relative dominance} (for ideals) and  \emph{subordination} (for
families of states), the main result being the Relative Dominance theorem (\ref{rel-abs-uniqueness}), which explains how the relative dominance (with respect to a $C^*$-subalgebra $B\subset A$) for ideals $J \triangleleft A$ 
is equivalent to the existence of certain sets of states on $B$. 

The remainder of the paper (of which the first two Sections review background information) is organized as follows. Section \ref{central-sec}  contains several general results pertaining to central inclusions
$C_0(Q)\subset A$, which are needed in the main application to \'etale group bundles in the subsequent section. Our analysis pays particular attention to points $q\in Q$ of 
\nagysecondrev{\em continuous reduction}, relative to a conditional expectation $\mathbf{E}:A\to C_0(Q)$. 
\nagyrev{In Section \ref{bundles-sec} we apply this methodology to the central 
 inclusion $C_0(\mathcal{G}^{(0)}\subset C^*_\text{red}(\mathcal{G})$, for 
an \'etale group bundle $\mathcal{G}$.} 

In Section \ref{min-sec} we provide a conceptual framework for {\em minimality\/} of $C^*$-inclusions, 
which is used in Proposition \ref{reg-simple}, where a general simplicity criterion is given. 
All these results are \nagyrev{applied to} groupoid $C^*$-algebras in Section \ref{grp-sec},
where we characterize simplicity for reduced groupoid $C^*$-algebras in terms of the interior isotropy subgroupoid
(Proposition \ref{red-simple-prop}). The main results in this section are Theorem \ref{simplicity-for-red}, which gives
a sufficient condition for simplicity of the reduced $C^*$-algebra of an \'etale
groupoid, along with Corollaries  \ref{red-simple-converse-other} and  Corollary \ref{red-simple-converse-amenable}, which provide some necessary conditions. Along the way (Theorem \ref{full-simple-thm}), we also recover the characterization of simplicity for the full groupoid $C^*$-algebra that previously appeared in \cite[Thm. 5.1]{BCFS}.

The paper concludes with  Section \ref{cross-prod-sec}, where we specialize our preceding results to \emph{transformation groupoids}, those which correspond to actions of discrete groups on locally compact Hausdorff spaces. We review (a slightly modified, but equivalent) construction of the transformation groupoid $C^*$-algebra $C_0(Q) \rtimes_r G$. With the help of
Theorem \ref{t-g-cont-thm}, the simplicity criterion from Corollary \ref{simplicity-for-cpred} 
(which sharpens \cite[Thm. 14(1)]{Ozawa})
has a nicer formulation than its sibling Theorem \ref{simplicity-for-red}.

In the Appendix we give technical results for embedding the (full) $C^*$-algebra of an open subgroupoid into the full $C^*$-algebra of an \'etale groupoid. 

\section{Preliminaries on Essential Inclusions}

\begin{notations}
Given a $C^*$-algebra $A$, the notation $J\unlhd A$ signifies that $J$ is a closed two-sided ideal in $A$. The instance when $J\unlhd A$ and $J\neq A$ is indicated by the notation $J\lhd A$. 

Given $C^*$-algebras $A$, $B$, to any positive linear map $\psi: A \to B$, one associates the following ideals:
\begin{itemize}
\item $L_\psi=\{a\in A\,:\,\psi(a^*a)=0\}$, the largest closed {\em left\/}  ideal in $A$, on which $\psi$ vanishes;
\item $K_\psi = \{a \in A: \psi(xay) =0,  \forall x,y \in A \}$, the largest closed {\em two-sided\/} ideal in $A$,
on which $\psi$ vanishes.
\end{itemize}
(In the scalar case $B=\mathbb{C}$ and $\psi\neq 0$, $K_\psi$ is nothing else but the kernel of the GNS representation $\Gamma_\psi$.) 

More generally, if $\Psi = \{\psi_i:A\to B_i\}_{i \in I}$ is a collection of positive linear maps between $C^*$-algebras, we let $L_\Psi=\bigcap_i L_{\psi_i}$ and $K_\Psi=\bigcap_i K_{\psi_i}$.
\end{notations}

\begin{definition}\label{ess-faithful-def}
A collection $\Psi = \{\psi_i : A \to B_i\}_{i \in I}$, as above, is said to be {\em jointly essentially faithful}, if $K_\Psi=\{0\}$.
This is a weaker notion than {\em joint (honest) faithfulness}, which requires
 $L_\Psi=\{0\}$.
(In the case of single maps, the term ``joint'' is omitted.)
%
\end{definition}


\begin{notations}
For a $C^*$-algebra $A$, we denote its \emph{state space\/} by $S(A)$, and we denote 
its \emph{pure state space\/} by $P(A)$.
\end{notations}

\begin{definition}\label{def-non-deg}
A $C^*$-inclusion $B\subset A$ is said to be \emph{non-degenerate}, if $B$ contains an approximate unit for $A$
(and consequently, {\em every\/} approximate unit for $B$ is also an approximate unit for $A$). 
More generally, a  $*$-homomorphism $\Phi:B\to A$ is said to be non-degenerate, if $\Phi(B)\subset A$ is a non-degenerate $C^*$-inclusion.
\end{definition}
\begin{remark}\label{CH}
Using the Cohen-Hewitt Factorization Theorem, non-degeneracy for a $C^*$-inclusion $B\subset A$ is equivalent to the equality
$A=BAB$, i.e. the fact that any $a\in A$ can be written as a product $a=b_1a'b_2$, with $a'\in A$, $b_1,b_2\in B$.
\end{remark}

With an eye on some further developments in Section \ref{central-sec}, the Remark below collects several useful features concerning multipliers.

\begin{remark} \label{MBsubMA-rem}
A non-degenerate inclusion $B\subset A$ always gives rise to a unital (thus non-degenerate) $C^*$-inclusion 
\begin{equation}
M(B)\subset M(A),
\label{MBsubMA1}
\end{equation}
by associating to any 
$m\in M(B)$ (and using Remark \ref{CH}), the left and right multiplication operators
$L_m:A\ni b_1ab_2\longmapsto (mb_1)ab_2\in A$ and
$R_m:A\ni b_1ab_2\longmapsto b_1a(b_2m)\in A$. (The parentheses surround elements in $B$.)
Equivalently, if $(u_\lambda)_\lambda\subset B$ is some approximate unit for $A$, then $L_ma=\lim_{\lambda,\mu}(u_\lambda m u_\mu)a$ and
$R_ma=\lim_{\lambda,\mu} a(u_\lambda m u_\mu)$.

The same argument shows that any non-degenerate $*$-homomorphism  $\Phi:B\to A$ extends uniquely to to a
unital $*$-homomorphism \
$$M\Phi:M(B)\to M(A),$$ 
which is continuous in the strict topology.

With the help of the inclusion \eqref{MBsubMA1}, the non-degeneracy of $B\subset A$ yields the following useful identifications
\begin{align}
\label{MBsubMA2}
M(B)&=\{m\in M(A)\,:\,mB\subset B\}=\{m\in M(A)\,:\,Bm\subset B\};\\
\label{MBsubMA3}
B&=M(B)\cap A.
\end{align}
The inclusion of the second set in \eqref{MBsubMA2} in the third one (and by symmetry, their equality) can justified again using an approximate unit $(u_\lambda)_\lambda\subset B$ for $A$, as follows. If $m\in M(A)$ satisfies
$mB\subset B$, then for any
$b\in B$ we have $bm=\lim_\lambda (bm)u_\lambda=\lim_\lambda b(mu_\lambda)\in B$.
The equality \eqref{MBsubMA3}, now follows immediately from \eqref{MBsubMA2}, because if $a\in M(B)\cap A$, then
$a=\lim_\lambda au_\lambda\in B$.
\end{remark}

\begin{notation}
A non-degenerate $C^*$-inclusion $B\subset A$ yields a restriction map
$$\mathbf{r}_{A\downarrow B}: S(A)\ni\varphi \longmapsto \varphi|_B\in S(B),$$ 
which (using the Hahn-Banach Theorem) is surjective.
\end{notation}
\begin{definition}
Given a non-degenerate inclusion $B\subset A$, and some subset $\Phi\subset S(B)$, we call a subset
$\Sigma\subset S(A)$ an {\em $A$-lift of $\Phi$}, if
$\mathbf{r}_{A\downarrow B}\big|_{\Sigma}$ is a bijection of $\Sigma$ onto $\Phi$. 
Equivalently, an $A$-lift is the range of a cross-section
$\boldsymbol{\sigma}:\Phi\to S(A)$ of $\mathbf{r}_{A\downarrow B}$ (i.e. a right inverse of 
$\mathbf{r}_{A\downarrow B}$ over $\Phi$), which allows us to enumerate
$\Sigma=\{\boldsymbol{\sigma}(\varphi)\}_{\varphi\in\Phi}$.

In the case when $\Phi\subset P(B)$, we will can an $A$-lift $\Sigma \subset P(A)$ a
a \emph{pure $A$-lift}. 
\end{definition}

\begin{remark}\label{extension-remark}
A standard application of Krein-Milman Theorem shows that any set $\Phi \subset P(B)$ always
has pure lifts. Furthermore, if $\Phi\subset P(B)$ has a {\em unique pure\/} $A$-lift $\Sigma$, then
$\Sigma$ is in fact the only possible $A$-lift, meaning that every $\varphi\in \Phi$ has a {\em unique},
extension to a state $\check\varphi\in S(A)$. In this vein, the ``Abstract'' Uniqueness Theorem from \cite[Thm 3.2]{BNRSW} has the following formulation.
\end{remark}

\begin{theorem}\label{abs1bnrsw} (cf. \cite[Thm. 3.2]{BNRSW})
Assume $B\subset A$ is a non-degenerate $C^*$-inclusion, and $\Phi\subset S(B)$ is a collection which has a unique $A$-lift
$\Sigma=\{\boldsymbol{\sigma}(\varphi)\}_{\varphi\in\Phi}\subset S(A)$. 
If $\Sigma$ is jointly essentially faithful, 
then:
\begin{itemize}
\item[{\sc (e)}] the only ideal $L\lhd A$, that satisfies $B\cap L=\{0\}$, is the zero ideal $L=\{0\}$.
\end{itemize}
{\rm A non-degenerate inclusion $B\subset A$ satisfying condition {\sc (e)} is called {\em essential}.}
\end{theorem}

In the spirit of the above theorem, we conclude this section with several technical results, which will be useful to us later, the first of which (stated only as a Remark) is very elementary.

\begin{remark}\label{embed-rem}
For a $*$-homomorphism between (non-zero) $C^*$-algebras $\pi:A\to B$, the following conditions are equivalent:
\begin{itemize}
\item[(i)] $\pi$ is injective;
\item[(ii)] there exists a non-empty set of positive linear maps $\Sigma=\{B\xrightarrow{\sigma_i} D_i\}_{i\in I}$, such that
the set of linear positive maps $\Sigma^\pi=\{A\xrightarrow{\sigma_i\circ\pi} D_i\}_{i\in I}$ is jointly essentially faithful.
\end{itemize}
Indeed, ``$(ii)\Rightarrow (i)$'' follows from the observation that
$\text{Ker}\,\pi\subset \nagyrev{K_{\Sigma^\pi}}$, while the
implication `` $(i)\Rightarrow (ii)$'' follows by letting $\Sigma=S(B)$ and using the fact that
 every state on the $C^*$-subalgebra $\pi(A)\subset B$ can be extended to a state on $B$, thus
(i) implies
$S(B)^\pi\supset S(A)$.
\end{remark}

\begin{lemma}\label{extend-lemma}
Assume $A$, $B$ are $C^*$-algebras, $A_0\subset A$ is a dense $*$-subalgebra, and $\pi_0:A_0\to B$ is a $*$-homomorphism. The following conditions are equivalent.
\begin{itemize}
\item[(i)] $\pi_0$ extends to a (necessarily unique) $*$-homomorphism $\pi:A\to B$;
\item[(ii)] there exists a jointly faithful set of positive linear maps
$\Sigma=\{B\xrightarrow{\sigma_i}D_i\}_{i\in I}$, such that all maps in the set
$\Sigma^{\pi_0}=\{A_0\xrightarrow{\sigma_i\circ\pi}D_i\}_{i\in I}$ are bounded (in the norm from $A$); equivalently, all maps in $\Sigma^{\pi_0}$ extend to linear positive maps on $A$.
\end{itemize}
\end{lemma}

\begin{proof}
The implication ``$(i)\Rightarrow (ii)$'' is obvious, by taking $\Sigma$ to be the whole state space $S(B)$ of $B$. For the implication ``$(ii)\Rightarrow (i)$,'' we first observe that, since joint faithfulness is preserved  by scaling the $\sigma_i$'s, we can assume that
$\sup_{i\in I}\max\left\{\|\sigma_i\|,\,\|\sigma_i\circ\pi_0\|\right\}<\infty$. \nagyrev{Under this additional assumption, the joint faithfulness of $\Sigma$ yields a single faithful positive linear
map $B\ni b\longmapsto (\sigma_i(b))_{i\in I}\in\prod_{i\in I}D_i$; in other words, we can assume that
$\Sigma$ is a singleton $\{B\xrightarrow{\sigma}D\}$}. Since the desired condition reads
$\|\pi_0(x^*x)\|\leq\|x^*x\|$, $\forall\,x\in A_0$, we can assume $A_0$ and $A$  are singly generated by a single
positive element $a_0$. In other words, in condition (ii) we can assume both $A$ and $B$ are separable (also abelian, if we want). This means that for our singleton set $\Sigma=\{B\xrightarrow{\sigma}D\}$ we can also assume $D$ is separable, so it has a faithful state $\psi$, which will then make $\psi\circ\sigma:B\to\mathbb{C}$ faithful as well, so in fact we can assume $D=\mathbb{C}$. With all these reductions in mind, our implication reduces to the following statement: {\em with $A_0$, $A$, $B$ and $\pi_0$ as above, if  $\sigma$ is a faithful state on $B$, such that the composition $\sigma\circ\pi_0:A_0\to\mathbb{C}$ is bounded (thus it extends to a positive linear functional $\phi:A\to\mathbb{C}$), then
$\|\pi_0(a_0)\|\leq\|a_0\|$, for any element $a_0\in A_0$, which is of the form $a_0=x^*x$ with $x\in A_0$}. 
(Thus $a_0$ is positive in $A$, and $\pi_0(a_0)$ is positive in $B$, as well). However, this implication follows immediately from the well known fact that, whenever $\sigma:B\to\mathbb{C}$ is a faithful positive linear functional, for any positive element $b\in B$, one has $\|b\|=\lim_{n\to\infty}\sigma(b^n)^{1/n}$. Applying this to $b=\pi_0(a)$ implies $\|\pi_0(a_0)\|=\lim_{n\to\infty}\phi(a_0^n)^{1/n}
\leq\limsup_{n\to\infty}\left(\|\phi\|\cdot\|a_0\|^n\right)^{1/n}=\|a_0\|$, and we are done.
\end{proof}

\begin{mycomment}
Condition (ii) above cannot be relaxed to essential faithfulness. For example, let $A=C([0,1])$, let 
$A_0\subset A$ be the $*$-subalgebra of polynomial functions, and let $B=M_2$ (the $2\times 2$ matrices). 
\nagyrev{On $A_0$, the norm from $A$ is $\|f\|_A=\sup_{t\in [0,1]}|f(t)|$. Since we can also view
$A_0$ as a $*$-subalgebra of $C([0,2])$, we can consider the
$*$-homomorphism $\pi_0:A_0\ni f\longmapsto \left(\begin{array}{cc}f(1) & 0 \\ 0 & f(2)\end{array}\right)\in M_2$, which is clearly not bounded in the norm $\|\,.\,\|_A$. However, the state
$\sigma:M_2\ni\left(\begin{array}{cc}a&b\\ c& d\end{array}\right)\longmapsto a\in\mathbb{C}$ is essentially faithful on $M_2$, and the composition $\sigma\circ \pi:f\longmapsto f(1)$ is bounded in the norm $\|\,.\,\|_A$}. 
\end{mycomment}

As a combination of the preceding two results, we now have the following embedding criterion.

\begin{proposition}\label{embed-crit}
Assume $A$, $B$ are $C^*$-algebras, $A_0\subset A$ is a dense $*$-subalgebra, and $\pi_0:A_0\to B$ is a $*$-homomorphism. Assume also we have a double $I$-tuple of positive linear maps
$\{B\xrightarrow{\sigma_i}D_i\xleftarrow{\psi_i}A\}_{i\in I}$, with the following properties.
\begin{itemize}
\item[$(a)$] the collection $\Sigma=\{B\xrightarrow{\sigma_i}D_i\}_{i\in I}$ is jointly faithful (on $B$);
\item[$(b)$] the collection $\Psi=\{A\xrightarrow{\psi_i}D_i\}_{i\in I}$ is jointly essentially faithful (on $A$);
\item[$(c)$] $\sigma_i\circ\pi_0=\psi_i\big|_{A_0}$, $\forall\,i\in I$.
\end{itemize}
Then $\pi_0$ extends (uniquely) to an injective $*$-homomorphism $\pi:A\to B$. In particular, $\Psi$ is jointly faithful.
\end{proposition}
\begin{proof}
By Lemma \ref{extend-lemma}, conditions $(a)$ and $(c)$ imply the fact that $\pi_0$ indeed extends to a (necessarily unique) $*$-homomorphism $\pi:A\to B$. \nagyrev{By Remark \ref{embed-rem}, condition $(b)$ ensures that $\pi$ is indeed injective.}
\end{proof}

\section{Preliminaries on Groupoid \texorpdfstring{$C^*$}{C*}-algebras}\label{grp-back}

\begin{convention}
All groupoids in this paper are assumed to be
second countable, locally compact and Hausdorff.
\end{convention}

\begin{notations}
If $\mathcal{G}$ is such a groupoid, then as usual, we denote its unit space by $\mathcal{G}^{(0)}$,
and we let $r,s:\mathcal{G}\to\mathcal{G}^{(0)}$ denote the range and source maps; and for each integer $n\geq 2$,
the set of composable $n$-tuples is 
$$\mathcal{G}^{(n)}=\big\{(\gamma_1,\gamma_2,\dots,\gamma_n)\in\textstyle{\prod_1^n\mathcal{G}}\,:\,r(\gamma_j)=s(\gamma_{j-1}),\,\,2\leq j\leq n\big\}.$$
Given some non-empty subset $\mathcal{X}\subset\mathcal{G}$, we denote the set
$\{\gamma^{-1}\,:\,\gamma\in\mathcal{X}\}$ simply by $\mathcal{X}^{-1}$.
Given non-empty subsets
$\mathcal{X}_1,\mathcal{X}_2,\dots,\mathcal{X}_n\subset\mathcal{G}$, we define
$$\mathcal{X}_1\mathcal{X}_2\dots\mathcal{X}_n=\big\{\gamma_1\gamma_2\cdots\gamma_n\,\,\big|\,\,
(\gamma_1,\dots,\gamma_n)\in\textstyle{\mathcal{G}^{(n)}\cap \prod_{j=1}^n\mathcal{X}_j}\big\}.
$$
When one of these sets is a singleton, say $\{\gamma\}$, we omit the braces; in particular, if $u\in\mathcal{G}^{(0)}$, the sets $u\mathcal{G}$ and $\mathcal{G}u$ are precisely the range and source fibers
$r^{-1}(u)$ and $s^{-1}(u)$, respectively, while $u\mathcal{G}u=u\mathcal{G}\cap \mathcal{G}u$ is the {\em isotropy group at $u$}.
\end{notations}

\begin{definition}
A groupoid $\mathcal{G}$ is 
{\em \'etale}, if $r$ and $s$ are local homeomorphsims; equivalently, there exists a basis for the topology on $\mathcal{G}$ consisting of {\em open bisections}, i.e. open subsets $\mathcal{B}\subset\mathcal{G}$,
for which
$r:\mathcal{B}\to r(\mathcal{B})$ and
$s:\mathcal{B}\to s(\mathcal{B})$ are homeomorphisms onto open sets in $\mathcal{G}$ (in particular,
$\mathcal{G}^{(0)}$ is clopen in $\mathcal{G}$). 
As a consequence, for each unit $u\in\mathcal{G}^{(0)}$ the sets $u\mathcal{G}$ and
$\mathcal{G}u$ are discrete in the relative topology; hence 
compact subsets of $\mathcal{G}$ intersect any one of these sets finitely many times.
\end{definition}

We now recall the construction (cf. \cite{Renault}) of the $C^*$-algebras associated to
an \'etale groupoid $\mathcal{G}$, as above.  One starts off by endowing $C_c(\mathcal{G})$
with the following $*$-algebra structure:
\begin{align*}
(f\times g)(\gamma) &= \sum_{(\alpha,\beta) \in \mathcal{G}^{(2)}: \alpha \beta = \gamma} f(\alpha) g(\beta);\\
f^*(\gamma)&=\overline{f(\gamma^{-1})}.
\end{align*}
\begin{nagyrevenv}
For future reference, we recall that these operations obey the following support rules:
\begin{align}
\text{supp}\,(f\times g)&\subset(\text{supp}\,f)(\text{supp}\,g);\label{sup-prod}\\
\text{supp}\,f^*&=(\text{supp}\,f)^{-1}.\label{sup-inv}
\end{align}
\end{nagyrevenv}
With these definitions, the {\em full\/} $C^*$-norm on $C_c(\mathcal{G})$ is given as
\begin{equation}
\|f\|_{\text{full}}=\sup\left\{\big\|\pi(f)\big\|\,:\,\pi\text{ $*$-representation of $C_c(\mathcal{G})$}\right\},
\label{full-norm}
\end{equation}
By a deep result -- Renault's Disintegration Theorem (\cite[Thm. II.1.21, Corollary II.1.22]{Renault}), the quantity $\|f\|_{\text{full}}$ is finite, for each $f\in C_c(\mathcal{G})$, and furthermore, satisfies
 $\|f\|_{\text{full}}\leq\|f\|_I$, where
$$\|f\|_I=\sup_{u\in\mathcal{G}^{(0)}}\max\big\{\textstyle{\sum_{\gamma\in u\mathcal{G}}|f(\gamma)|,\,
\sum_{\gamma\in \mathcal{G}u}|f(\gamma)}|\big\},\,\,\,f\in C_c(\mathcal{G}).$$

As $\mathcal{G}^{(0)}$ is clopen in $\mathcal{G}$, we have an inclusion $C_c(\mathcal{G}^{(0)})\subset C_c(\mathcal{G})$, which turns $C_c(\mathcal{G}^{(0)})$ into a 
$*$-subalgebra; however, the $*$-algebra operations on $C_c(\mathcal{G}^{(0)})$ inherited from
$C_c(\mathcal{G})$ coincide with the usual (pointwise) operations:
$h^*=\bar{h}$ and $h\times k=hk$, $\forall\,h,k\in C_c(\mathcal{G}^{(0)})$. In fact, something similar can be said concerning 
the left and right
$C_c(\mathcal{G}^{(0)})$-module structure of $C_c(\mathcal{G})$: for all $f\in C_c(\mathcal{G})$, 
$h\in C_c(\mathcal{G}^{(0)})$ we have
\begin{align}
(f\times h)(\gamma)&=f(\gamma)h\big(s(\gamma)\big);\\
(h\times f)(\gamma)&=h\big(r(\gamma)\big)f(\gamma).
\end{align} 
When restricted to $C_c(\mathcal{G}^{(0)})$, the full $C^*$-norm agrees with the usual sup-norm $\|\cdot\|_\infty$, so by completion,
the embedding $C_c(\mathcal{G}^{(0)})\subset C_c(\mathcal{G})$ gives rise to a non-degenerate inclusion
$C_0(\mathcal{G}^{(0)})\subset C^*(\mathcal{G})$. 

The restriction map
$C_c(\mathcal{G})\ni f\longmapsto f|_{\mathcal{G}^{(0)}}\in C_c(\mathcal{G}^{(0)})$ 
(which is contractive in the full norm) gives rise by completion to 
 a {\em conditional expectation of $C^*(\mathcal{G})$ onto $C_0(\mathcal{G}^{(0)})$}, i.e. a norm one linear map
 $\mathbb{E}:C^*(\mathcal{G})\to C_0(\nagyrev{\mathcal{G}}^{(0)})$,
such that 
\begin{itemize}
\item $\mathbb{E}(f)=f$, $\forall\,f\in C_0(\mathcal{G}^{(0)})$;
\item $\mathbb{E}(f_1af_2)=f_1\mathbb{E}(a)f_2$, $\forall\,a\in C^*(\mathcal{G}),\,f_1,f_2\in C_0(\mathcal{G}^{(0)})$;
\end{itemize}

The KSGNS representation (\cite{Lance}) associated with $\mathbb{E}$ is a $*$-homomorphism 
$\boldsymbol{\Lambda}_{\mathbb{E}}:
C^*(\mathcal{G})\to\mathcal{L}\left(L^2\left(C^*(\mathcal{G}),\mathbb{E}\right)\right)$, where 
$L^2\left(C^*(\mathcal{G}),\mathbb{E}\right)$ is the Hilbert $C_0(\mathcal{G}^{(0)})$-module obtained by completing
$C^*(\mathcal{G})$ in the norm given by the inner product
$\langle a|b\rangle_{\mathbb{E}}=\mathbb{E}(a^*b)$. The kernel of this representation is precisely the ideal $K_{\mathbb{E}}$ defined in Section 1. The quotient
$C^*(\mathcal{G})/K_\mathbb{E}$ is the so-called {\em reduced\/} groupoid $C^*$-algebra, denoted by $C^*_{\text{red}}(\mathcal{G})$. The ideal $K_{\mathbb{E}}$ \nagyrev{can} be described alternatively with the help of the usual
GNS representations $\Gamma_{\text{ev}_u\circ\mathbb{E}}:C^*(\mathcal{G})\to
\mathscr{B}(L^2(C^*(\mathcal{G}),\text{ev}_u\circ\mathbb{E}))$, $u\in\mathcal{G}^{(0)}$.
(The Hilbert space $L^2(C^*(\mathcal{G}),\text{ev}_u\circ\mathbb{E})$ is the completion of
$C_c(\mathcal{G})$ in the norm given by the inner product $\langle f|g\rangle_u=
(f^*\times g)(u)$.)
With these (honest) representations in mind, we have
$K_{\mathbb{E}}=\bigcap_{u\in\mathcal{G}^{(0)}}K_{\text{ev}_u\circ\mathbb{E}}$.

As was the case with the full groupoid $C^*$-algebra, after composing with the quotient map
$\pi_{\text{red}}:C^*(\mathcal{G})\to C^*_{\text{red}}(\mathcal{G})$, we still have an embedding 
$C_c(\mathcal{G})\subset C^*_{\text{red}}(\mathcal{G})$, so we can also view $C^*_{\text{red}}(\mathcal{G})$ as the completion of the convolution $*$-algebra $C_c(G)$ with respect to a (smaller) $C^*$-norm, denoted $\|\cdot\|_{\text{red}}$.
Again, when restricted to $C_c(\mathcal{G}^{(0)})$, the norm $\|\,\cdot\,\|_{\text{red}}$ agrees with
$\|\,\cdot\,\|_\infty$, so $C_0(\mathcal{G}^{(0)})$ still embeds in $C^*_{\text{red}}(\mathcal{G})$, and furthermore, since the natural expectation 
$\mathbb{E}$ vanishes on $K_{\mathbb{E}}$, we will have a reduced version of natural expectation, denoted by
 $\mathbb{E}_{\text{red}}:C^*_{\text{red}}(\mathcal{G})\to C_0(\mathcal{G}^{(0)})$, which satisfies 
$\mathbb{E}_{\text{red}}\circ\pi_{\text{red}}=\mathbb{E}$. Not only we know that $\mathbb{E}_{\text{red}}$ is essentially faithful (because $K_{\mathbb{E}}=\text{ker}\,\pi_{\text{red}}$), but in fact it is
(honestly) faithful on $C^*_{\text{red}}(\mathcal{G})$.

\section{Relative Dominance}\label{rel-dom-sec}

As it turns out, proving that an inclusion $B\subset A$ is essential is always tied up with some
analysis of
 $A$-lifts (i.e. {\em state extensions\/}) for states on $B$, as 
Corollary \ref{mother} below (which also contains some type of converse of Theorem \ref{abs1bnrsw}) 
 will demonstrate.
In preparation for these results, we start off by formulating relative versions for essential faithfulness, as well as property {\sc (e)} from Theorem \ref{abs1bnrsw}.
  
\begin{definitions}\label{def-sub-dom}
Suppose a $C^*$-algebra $A$ is given, along with some ideal 
$J \lhd  A$.
\begin{itemize}
\item[(i)] A collection $\Psi = \{\psi_i:A\to B_i\}_{i \in I}$ of positive linear maps between $C^*$-algebras
is said to be {\em subordinated to $J$}, if $K_\Psi\subset J$.
\item[(ii)] For a non-degenerate $C^*$-inclusion $B\subset A$, we say that $J$ is {\em dominant relative to $B$}, if
{\em whenever an ideal $L\lhd A$ satisfies $B\cap L\subset B\cap J$, it also satisfies
$L\subset J$.}
\end{itemize}
\end{definitions}

\begin{remark}
When specializing the above Definition to the zero ideal, it is straightforward that
\begin{itemize}
\item[(i)] A family $\Psi$ as above is jointly essentially faithful, if and only if
$\Psi$ is subordinated to the zero ideal $\{0\}$.
\item[(ii)] A a non-degenerate $C^*$-inclusion $B\subset A$ is {\em essential}, if and only if
the zero ideal $\{0\}$ is dominant relative to $B$.
\end{itemize}
\end{remark}

\begin{theorem}\label{rel-abs-uniqueness}\label{big-dominant}
Suppose that a non-degenerate $C^*$-inclusion $B \subset A$ is given, along with some ideal $J \lhd A$.
Regard the state space $S(B / B \cap J)$ as a subset of $S(B)$ and the 
pure state space $P(B / B \cap J)$ as a subset of $P(B)$ (see the Note preceding the proof below).
The following are equivalent: 
\begin{itemize}
\item[(i)] $J$ is dominant relative to $B$; 
\item[(ii)] there exists $\Phi \subset S(B /B \cap J)$ such that all $A$-lifts of $\Phi$ are subordinated to $J$; 
\item[(iii)] there exists $\Phi \subset P(B/B \cap J)$ such that all $A$-lifts of $\Phi$ are subordinated to $J$; 
\item[(iii')] there exists $\Phi \subset P(B/B \cap J)$ such that all pure $A$-lifts of $\Phi$ are subordinated to $J$.
\end{itemize}
\end{theorem}
\begin{note}
The inclusions mentioned in the statement are simply the compositions with the
quotient $*$-homomorphism $\rho:B\to B/B\cap J$, namely the map
$S(B/B\cap J)\ni\varphi\longmapsto \varphi\circ\rho\in S(B)$ and its restriction to
$P(B/B\cap J)$. These inclusions simply identify
$S(B/B\cap J)$ with the set $\{\varphi\in S(B)\,:\,\varphi\big|_{B\cap J}=0\}$ and
$P(B/B\cap J)$ with the set $\{\varphi\in P(B)\,:\,\varphi\big|_{B\cap J}=0\}$.
\end{note}
\begin{proof}
\noindent (i)$\Rightarrow$(iii). Assume that $J$ is dominant relative to $B$, and let us show that property (iii) holds for $\Phi = P(B / B \cap J)$. Fix for the moment some $A$-lift
$\Sigma = \{\boldsymbol{\sigma}(\varphi)\}_{\varphi \in P(B/B \cap J)}$ for $P(B/B\cap J)$
(associated with some cross-section $\boldsymbol{\sigma}:P(B/B\cap J)\to S(A)$), and let us justify the inclusion
$K_\Sigma \subset J$. Using the assumption (i), it suffices to prove the inclusion
$B\cap K_{\Sigma}\subset B\cap J$. To this end, fix some
$b\in B\cap K_{\Sigma}$, and let us show that $b\in B\cap J$. By our assumption on $b$, we know that
$$(\boldsymbol{\sigma}(\varphi))(xby)=0,\,\,\,\forall\,x,y\in A,\,\varphi\in P(B/ B\cap J).$$
By taking $x$ and $y$ to be terms in an approximate unit for $B$, and using the cross-section condition
$\boldsymbol{\sigma}(\varphi)\big|_B=\varphi$, the above equalities imply
$$\varphi(b)=0,\,\,\,\forall\,\varphi\in P(B/B\cap J),$$
which then clearly implies that $b$ indeed belongs to $B\cap J$.
 
It is obvious that (ii)$\Leftarrow$(iii)$\Rightarrow$(iii'), so it stands to prove that 
[(ii) or (iii')]$\Rightarrow$(i). Start off by assuming the existence of a collection $\Phi\subset S(B/B\cap J)$, which either
\begin{itemize}
\item[$(*)$] satisfies (ii), or
\item[$(**)$] is a subset of $P(B / B \cap J)$ and satisfies (iii'),
\end{itemize} 
let $L\lhd A$ be an ideal satisfying $B\cap L\subset B\cap J$, and let us prove the inclusion $L \subset J$.

Consider the quotient $\pi: A \to A/L$ and the non-degenerate $C^*$-subalgebra 
$$\pi(B)(\simeq B/B\cap L)\subset A/L.$$
Fix a cross-section $\boldsymbol{\eta}: S(\pi(B))\to S(A/L)$ for $\mathbf{r}_{(A/L)\downarrow \pi(B)}$, which
maps $P(\pi(B))$ into $P(A/L)$. (Use Remark \ref{extension-remark}.) In other words, for each state $\varphi\in S(B)$ which vanishes on $B\cap L$, we choose $\boldsymbol{\eta}(\varphi)$ to be an extension of $\varphi$ to a state on $A$, which vanishes on $L$, and furthermore, in case $\varphi$ were pure, $\boldsymbol{\eta}(\varphi)$ is also chosen to be pure.
 
By the assumption $B \cap L \subset B \cap J$, it follows that $S(B/B \cap J)\subset S(B/B\cap L)$
and $P(B/B \cap J)\subset P(B/B\cap L)$ (that is, if a state $\varphi$ on $B$ vanishes on $B\cap J$, then it also
vanishes on $B\cap L$).
This means that we can view $\Phi\subset S(B/B\cap L)$ (or, in case $(**)$, we can view
$\Phi\subset P(B/B\cap L)$), thus $\Sigma=\{\boldsymbol{\eta}(\varphi)\}_{\varphi\in\Phi}$ defines an $A$-lift for 
$\Phi$, which is pure in case $(**)$. Therefore, by \nagyrev{either (ii) or (iii')}, it follows that
$\Sigma$ is subordinated to $J$, i.e.
$\bigcap_{\varphi\in\Phi} K_{\boldsymbol{\eta}(\varphi)}\subset J$. However, by our definition of
$\boldsymbol{\eta}$, we know that, for each $\varphi\in\Phi$, the state
$\boldsymbol{\eta}(\varphi)$ vanishes on $L$. In other words, 
$L\subset K_{\boldsymbol{\varphi}}$, $\forall\,\varphi\in \Phi$, which then implies
$L\subset\bigcap_{\varphi\in\Phi} K_{\boldsymbol{\eta}(\varphi)}\subset J$, and we are done.
\end{proof}

\begin{note}
If some $\Phi \subset S(B/B \cap J)$ satisfies $(2)$, then any larger set $\Psi \supset \Phi$ of states on $B/B \cap J$ will satisfy it as well. Similarly, the class of all subsets $\Phi \subset P(B/B \cap J)$ satisfying (3) or (3') is closed upward. 
\end{note}

If in the preceding Theorem we let $J=\{0\}$, we obtain the following generalization of Theorem \ref{abs1bnrsw}.

\begin{corollary}\label{mother}
For a non-degenerate inclusion $B \subset A$, the following are equivalent: 
\begin{itemize}
\item[(i)] $B \subset A$ is essential;
\item[(ii)] there exists $\Phi \subset S(B)$ such that all $A$-lifts of $\Phi$ are jointly essentially faithful; 
\item[(iii)] there exists $\Phi \subset P(B)$ such that all $A$-lifts of $\Phi$ are jointly essentially faithful; 
\item[(iii')] there exists $\Phi \subset P(B)$ such that all pure $A$-lifts of $\Phi$ are jointly essentially faithful. 
\end{itemize}
\end{corollary}

\section{Central Inclusions}\label{central-sec}

The main applications of Theorem \ref{rel-abs-uniqueness} and its Corollary \ref{mother} are Theorem \ref{simple-dom} 
and its Corollary \ref{simple-essential} below, which deal with
{\em central inclusions} $B\subset A$, i.e. those for which $B$ is contained in the center of $A$, that is, $ba=ab$, $\forall\,a\in A,\,b\in B$. Such inclusions force $B$ to be abelian, so we can identify $B=C_0(Q)$, for some locally compact space $Q$.

\nagyrev{A large class of examples -- which the results in this section are tailored for -- 
consists of those central inclusions of the form 
$C_0(\mathcal{G}^{(0)})\subset C^*_\text{red}(\mathcal{G})$, arising from {\em \'etale group bundles\/}; they are
 discussed in Section \ref{bundles-sec}.}

\begin{notations}
Assume $C_0(Q)\subset A$ is a central non-degenerate inclusion. For any point $q\in Q$, denote by $J^{\text{unif}}_q\lhd A$ the ideal
generated by $C_{0,q}(Q)=\{f\in C_0(Q)\,:\,f(q)=0\}$. With the help of $J^{\text{unif}}_q$, 
we can 
endow $A$ with the $C^*$-seminorms
$p^{\text{unif}}_q$, $q\in Q$, given by
\begin{align}
p^{\text{unif}}_q(a) & = ||a+J^{\text{unif}}_q||_{A/J^{\text{unif}}_q} = \inf \{ ||a+x||: x \in J^{\text{unif}}_q\} 
\notag \\
& = \inf \{ ||fa||: f \in C_{0,q}(Q), 0 \leq f \leq 1=f(q) \}.\label{punif=inf}
\end{align}
\end{notations}

\begin{nagyrevenv}
\begin{mycomment}
Central inclusions are special cases of so-called {\em $C_0(Q)$-algebras}, for which the condition that $C_0(Q)$ is a $C^*$-subalgebra of $A$ is relaxed to the condition that $C_0(Q)$ is contained in the center $Z(M(A))$ of the multiplier algebra $M(A)$, combined with the (non-degeneracy) condition that an approximate unit for $C_0(Q)$ converges to 
$1\in M(A)$ in the strict topology. In this general setting the seminorms $p^{\text{unif}}_q$ still make sense, when defined by the last equality above; we then can define our ideals by $J^{\text{unif}}_q=\text{ker}\,p^{\text{unif}}_q$. Whether we work with arbitrary $C_0(Q)$-algebras, or just with those special ones considered above, one can show (see e.g. \cite[Lemma 1.12]{Nagy}, or \cite[Appendix C]{Williams2}) that for each $a\in A$, the map
\begin{equation}\label{punif-map}
Q\ni q\longmapsto p^{\text{unif}}_q(a)\in [0,\infty)
\end{equation}
is {\em upper semicontinuous}, in the sense that
\begin{equation}
\limsup_{q\to q_0}p^{\text{unif}}_q(a)\leq
p^{\text{unif}}_{q_0}(a),\,\,\,\forall\,q_0\in Q.
\label{p=usc}
\end{equation}
(An equivalent characterization of \eqref{p=usc} is the fact that all sets
$\{q\in Q\,:\,p^{\text{unif}}_q(a)<s\}$, $s>0$, are open.)
Likewise (using non-degeneracy),  it is also pretty clear that
\begin{equation}
\lim_{q\to\infty}p^{\text{unif}}_q(a)=0,\,\,\,\forall\,a\in A.
\label{p-nondeg}
\end{equation}
Regardless of whether \eqref{punif-map} is continuous or not, by non-degeneracy we always have
\begin{equation}
\|a\|=\sup_{q\in Q}p^{\text{unif}}_q(a),\,\,\,\forall\,a\in A.
\label{norm-sup-punif}
\end{equation}
\end{mycomment}
\end{nagyrevenv}

\begin{theorem}\label{simple-dom}
Suppose that $C_0(Q) \subset A$ is a central non-degenerate inclusion, and assume (using the above notation) the set
$$Q^{\text{\rm unif}}_{\text{\rm simple}}=\{q\in Q\,:\,A/J^{\text{\rm unif}}_q\text{ is a simple $C^*$-algebra\/}\}$$
is non-empty. Then for any non-empty subset $Q_0\subset Q^{\text{\rm unif}}_{\text{\rm simple}}$, the ideal
$
\bigcap_{q\in Q_0}J^{\text{\rm unif}}_q$ is dominant relative to $C_0(Q)$.
\end{theorem}

\begin{proof}
For every $q\in Q$, let $\text{ev}_q$ denote the evaluation map
$C_0(Q)\ni f\longmapsto f(q)\in\mathbb{C}$, which defines a pure state on $C_0(Q)$.
The main step in our proof is contained in the following \nagyrev{claim.}
\begin{claim}
If $q\in Q^{\text{\rm unif}}_{\text{\rm simple}}$ and $\psi_q\in S(A)$ is any extension of $\text{\rm ev}_q$ to a
state on $A$, then $K_{\psi_q}=J^{\text{\rm unif}}_q$.
\end{claim}
First of all, since $K_{\psi_q}$ is the kernel of a GNS representation (associated to a state on $A$), it is trivial that
$K_{\psi_q}\subsetneq A$, so by the assumed simplicity of $A/J^{\text{unif}}_q$, it suffices to prove the inclusion
$J^{\text{unif}}_q\subset K_{\Psi_q}$. Secondly, since by construction, the set $\{f\in C_{0,q}(Q)\,:\,0\leq f\leq 1\}$ (equipped with the usual order) constitutes an approximate unit for $J^{\text{unif}}_q$, all we have to justify is the inclusion $C_{0,q}(Q)\subset
K_{\psi_q}$, which amounts to showing that
\begin{equation}
\psi_q(a^*f^*fa)=0,\,\,\,\forall\,f\in C_{0,q}(Q),\,a\in A.\label{simple-dom-claim}
\end{equation}
However, since the inclusion $C_0(Q)\subset A$ is central, for any $f\in C_{0,q}(Q)$, $a\in A$, we have
$$0\leq \psi_q(a^*f^*fa)=
\psi_q(f^*a^*af)\leq\psi_q(\|a\|^2f^*f)=\text{ev}_q(\|a\|^2f^*f)=\|a\|^2\cdot|f(q)|^2=0,$$
and \eqref{simple-dom-claim} follows.

Having proved the Claim, all we have to do is to show that the collection
$\Phi=\{\text{ev}_q\}_{q\in Q_0}$ satisfies condition (ii) from Theorem \ref{rel-abs-uniqueness}, when applied to the ideal 
$\bigcap_{q\in Q_0}J^{\text{\rm unif}}_q$.
But if we start with a $A$-lift $\Psi=\{\psi_q\}_{q\in Q_0}$ for $\Phi$ (by way of notation assuming $\psi_q\big|_{C_0(Q)}=
\text{ev}_q$, $\forall\,q\in Q_0$), then by the Claim we know that
$K_{\psi_q}=J^{\text{unif}}_q$, $\forall\,q\in Q_0$, which clearly implies
$K_\Psi=\bigcap_{q\in Q_0}K_{\psi_q}=\bigcap_{q\in Q_0}J^{\text{\rm unif}}_q.$
\end{proof}

\begin{corollary}\label{simple-essential}
Suppose $C_0(Q)\subset A$ is a central non-degenerate inclusion.
If, using the notations from Theorem \ref{simple-dom}, 
the set $Q^{\text{\rm unif}}_{\text{\rm simple}}$ is non-empty, and satisfies
\begin{equation}
\bigcap_{q\in Q^{\text{\rm unif}}_{\text{\rm simple}}}J^{\text{\rm unif}}_q=\{0\},
\label{simple-essential-eq}
\end{equation}
then the inclusion $C_0(Q)\subset A$ is essential.
\end{corollary}

\begin{nagyrevenv}
When trying to check the hypotheses in Corollary \ref{simple-essential}, a reasonable guess for a sufficient condition would be the {\em density\/} of $Q^{\text{unif}}_{\text{simple}}$ in $Q$. However, in part due to lack of continuity
of some of the maps \eqref{punif-map}, particularly if one of the inequalities \eqref{p=usc} is strict, one cannot expect this to be the case, as illustrated by the following.

\begin{example}
Let $A=C([0,2])$, and consider the inclusion $C([0,1])\subset A$ defined by extending every continuous function
$f:[0,1]\to\mathbb{C}$ to a continuous function on $[0,2]$, which is constant on $[1,2]$.
Clearly, $J^{\text{unif}}_q=\{f\in C([0,2])\,:\,f(q)=0\}$, for every $q\in [0,1)$, but
$J^{\text{unif}}_1=\{f\in C([0,2])\,:\,f\big|_{[1,2]}=0\}$, so
$Q^{\text{unif}}_{\text{simple}}=[0,1)$ is indeed dense in $Q=[0,1]$, but
the ideal $J^{\text{unif}}_{[0,1)}=\{f\in C([0,2])\,:\,f\big|_{[0,1]}=0\}$ is not $\{0\}$.
Coincidentally, the same ideal $J^{\text{unif}}_{[0,1)}$ satisfies $J^{\text{unif}}_{[0,1)}\cap C([0,1])=\{0\}$, which
also means that $C([0,1])\subset A$ is non-essential.
\end{example}
\end{nagyrevenv}

On way to ameliorate the above complication -- as shown in Theorem \ref{simple-essential-CE} below, as well as the failure of continuity for the maps \eqref{punif-map} -- as seen in Proposition \ref{cont-prop} below, is to use {\em conditional expectations}.

\begin{notations}
Assume $C_0(Q)\subset A$ is a non-degenerate central inclusion, and $\mathbf{E}$
 is a conditional expectation of  $A$ onto $C_0(Q)$.
For every $q\in Q$, consider the composition
$\text{ev}_q\circ\mathbf{E}$, which is a state on $A$, and its associated GNS representation
$\Gamma_{\text{ev}_q\circ\mathbf{E}}:A\to\mathscr{B}(L^2(A,\text{ev}_q\circ\mathbf{E}))$. 
The associated $C^*$-seminorm 
$A\ni a\longmapsto\left\|\Gamma_{\text{ev}_q\circ\mathbf{E}}(a)\right\|\in [0,\infty)$ will be denoted by
$p^{\mathbf{E}}_q$.
\end{notations}

\begin{remark}
Since $\mathbf{E}$ is $C_0(Q)$-linear, for every $q\in Q$ we have
$$\text{ev}_q\circ\mathbf{E}(fa)= 
\text{ev}_q\circ\mathbf{E}(af)=0,\,\,\,\forall\,f\in C_{0,q}(Q),\,a\in A,$$
thus $J^{\text{unif}}_q$ is contained in the kernel of the GNS representation
$\Gamma_{\text{ev}_q\circ \mathbf{E}}$, which means that
\begin{equation}
\label{pEleqpu}
p^{\mathbf{E}}_q(a)\leq p^{\text{unif}}_q(a),\,\,\,\forall\,q\in Q,\,a\in A.
\end{equation}
\end{remark}

\begin{theorem}\label{simple-essential-CE}
Assume $C_0(Q)\subset A$ is a non-degenerate central inclusion, such that $Q^{\text{\rm unif}}_{\text{\rm simple}}$ is dense in $Q$. If there exists an essentially faithful conditional expectation of $\mathbf{E}:A\to C_0(Q)$, then the inclusion $C_0(Q)\subset A$ is essential.
\end{theorem}
\begin{proof}
Consider the ideal 
$$H=\bigcap_{q\in Q^{\text{unif}}_{\text{simple}}}K_{\text{ev}_q\circ \mathbf{E}}=
\bigcap_{q\in Q^{\text{unif}}_{\text{simple}}}\text{ker}\,p^{\mathbf{E}}_q$$
(i.e. the intersection of the kernel of the GNS representations $\Gamma_{\text{ev}_q\circ\mathbf{E}}$,
$q\in Q^{\text{unif}}_{\text{simple}}$). 
On the one hand, by \eqref{pEleqpu}, it follows that
$$\bigcap_{q\in Q^{\text{unif}}_{\text{simple}}}J^{\text{unif}}_q\subset H,$$
so in order to draw the desired conclusion, by Corollary 
\ref{simple-essential} it suffices to show that $H$ is the zero ideal.

On the other hand, by construction, for every $x\in H$, the function
$\mathbf{E}(x)\in C_0(Q)$ vanishes on $Q^{\text{unif}}_{\text{simple}}$, which is dense in $Q$, thus by continuity
$\mathbf{E}(x)=0$. In other words, $H$ is a closed two-sided ideal contained in $\text{ker}\,\mathbf{E}$, so by
 the essential faithfulness of $\mathbf{E}$, $H$ is indeed zero ideal.
\end{proof}

\begin{remark}
The required essential faithfulness of $\mathbf{E}$ should not be regarded as too stringent, since for an essential 
non-degenerate inclusion $B\subset A$, it follows that {\em all conditional expectations $\mathbf{E}:A\to B$ (if any) are essentially faithful}.
\end{remark}

\begin{remark}
By definition, it is pretty obvious that, for a non-degenerate central inclusion $C_0(Q)\subset A$, the 
\nagyrev{essential} faithfulness of a conditional expectation
 $\mathbf{E}:A\to C_0(Q)$ is equivalent to:
\begin{equation}
\|a\|=\sup_{q\in Q}p^\mathbf{E}_q(a),\,\,\,\forall\,a\in A.
\label{norm-sup-pE}
\end{equation}
\end{remark}

As to the usefulness of conditional expectations in the issues dealing with \nagyrev{continuity properties of} the map \eqref{punif-map}, we have the
result below, which can be traced to \cite[Lemma 1.11]{Nagy}; due to its elementary nature, we include a proof for the benefit of the reader.

\begin{lemma}\label{lsc}
Suppose $C_0(Q)\subset A$ is a non-degenerate central inclusion, and $\mathbf{E}:A\to C_0(Q)$ is a conditional expectation.
For every $a\in A$, the map
\begin{equation}\label{pE-map}
Q\ni q\longmapsto p^{\mathbf{E}}_q(a)\in [0,\infty)
\end{equation}
is lower semicontinuous, in the sense that
\begin{equation}
\liminf_{q\to q_0}p^{\mathbf{E}}_q(a)\geq
p^{\mathbf{E}}_{q_0}(a),\,\,\,\forall\,q_0\in Q.
\label{p=lsc}
\end{equation}
\end{lemma}
\begin{proof}
An equivalent characterization of \eqref{p=lsc} is the fact that all sets
$D_s=\{q\in Q\,:\,p^{\mathbf{E}}_q(a)>s\}$, $s>0$, are open. 
Fix $a\in A$, $s>0$, as well some $q_0\in D_s$, and let us justify that $q_0$ is also in the interior of $D_s$.

By the definition of the GNS representation $\Gamma_{\text{ev}_q\circ \mathbf{E}}$, the condition
$p^{\mathbf{E}}_{q_0}(a) >s$ implies the existence of some $x\in A$, such that
$$
\text{ev}_{q_0}\circ\mathbf{E}(x^*a^*ax) > s^2\cdot \text{ev}_{q_0}\circ\mathbf{E}(x^*x)>0.$$
Since $\mathbf{E}$ is $C_0(Q)$-valued, by continuity, it follows that there exists a whole neighborhood $V$ of $q_0$ in 
$Q$, such that
$$
\text{ev}_{q}\circ\mathbf{E}(x^*a^*ax) > s^2\cdot\text{ev}_{q}\circ\mathbf{E}(x^*x)>0,\,\,\,\forall\,q\in V,$$
which in turn implies
$p^{\mathbf{E}}_{q}(a) >s$, $\forall\,q\in V$.
\end{proof}

\begin{nagyrevenv}
In preparation for Proposition \ref{cont-prop} below, we identify the following
uniqueness property of the seminorms $p^\text{unif}_q$,
in conjunction with \eqref{norm-sup-punif} and \eqref{p=usc}.

\begin{lemma}\label{usc-uniqueness}
Assume $C_0(Q)\subset A$ is a central non-degenerate inclusion, and $(\mathbf{p}_q)_{q\in Q}$ is a family of $C^*$-seminorms on $A$, satisfying
\begin{itemize}
\item[$(a)$] $\|a\|=\sup_{q\in Q}\mathbf{p}_q(a)$, $\forall\,a\in A$;
\item[$(b)$] $\mathbf{p}_q(fa)=|f(q)|\cdot \mathbf{p}_q(a)$, $\forall\,a\in A$, $f\in C_0(Q)$.
\end{itemize}
If $a\in A$, $q_0\in Q$ satisfy
\begin{equation}
\limsup_{q\to q_0}\mathbf{p}_q(a)\leq\mathbf{p}_{q_0}(a),\label{boldp=usc}
\end{equation}
then $\mathbf{p}_{q_0}(a)=p^\text{\rm unif}_{q_0}(a)$.
\end{lemma}

\begin{proof}
First of all, using the hypotheses $(a)$ and $(b)$, for any $f\in C_0(Q)$ with $0\leq f\leq 1=f(q_0)$, we have
$$\|fa\|\geq \mathbf{p}_{q_0}(fa)=|f(q_0)|\cdot\mathbf{p}_{q_0}(a)=\mathbf{p}_{q_0}(a),$$
so by \eqref{punif=inf}, we clearly have the inequality $p^\text{unif}_{q_0}(a)\geq 
\mathbf{p}_{q_0}(a)$. For the other inequality, fix for the moment $\varepsilon>0$ and use \eqref{boldp=usc} to produce some open set
$V_\varepsilon\ni q_0$, such that $\mathbf{p}_q(a)\leq \mathbf{p}_{q_0}(a)+\varepsilon$, $\forall\,q\in V$.
Fix also some $f_\varepsilon\in C_c(V)$ with $0\leq f_\varepsilon\leq 1=f_\varepsilon(q_0)$, and observe that, for every $q\in Q$, we have
$$\mathbf{p}_q(f_\varepsilon a)=
f_\varepsilon(q)\cdot \mathbf{p}_q(a)
\leq f_\varepsilon(q)\left(\mathbf{p}_{q_0}(a)+\varepsilon\right)
\leq \mathbf{p}_{q_0}(a)+\varepsilon,$$
which by $(a)$ implies $\|f_\varepsilon a\|\leq \mathbf{p}_{q_0}(a)+\varepsilon$, thus by
\eqref{punif=inf} we obtain $p^\text{unif}_{q_0}(a)\leq\mathbf{p}_{q_0}(a)+\varepsilon$, and we are done by letting 
$\varepsilon\to 0$.
\end{proof}
\end{nagyrevenv}

\begin{proposition}\label{cont-prop}
Assume $C_0(Q)\subset A$ is a central non-degenerate inclusion and $\mathbf{E}:A\to C_0(Q)$ is an essentially 
faithful conditional expectation. For a point $q_0$, the following conditions are equivalent:
\begin{itemize}
\item[(i)] the function \eqref{pE-map} is continuous at $q_0$, for each $a\in A$;
\item[(i')] there is some dense subset $A_0\subset A$, such that the function
 \eqref{pE-map} is continuous at $q_0$, for each $a\in A_0$;
\item[(i'')] there is some dense subset $A_0\subset A$, such that
$\limsup_{q\to q_0}p^\mathbf{E}_q(a)\leq p^\mathbf{E}_{q_0}(a)$, $\forall\,a\in A_0$;
\item[(ii)] $p^\text{\rm unif}_{q_0}(a)=p^\mathbf{E}_{q_0}(a)$, $\forall\, a\in A$;
\item[(ii')] there is some dense subset $A_0\subset A$, such that
$p^\text{\rm unif}_{q_0}(a)=p^\mathbf{E}_{q_0}(a)$, $\forall\, a\in A_0$.
\end{itemize}
\end{proposition}

\begin{proof}
The equivalences $(i)\Leftrightarrow (i')\Leftrightarrow (i'')$ and
$(ii)\Leftrightarrow (ii')$ are obvious, from \eqref{norm-sup-punif} and \eqref{norm-sup-pE}.

The implication $(ii)\Rightarrow (i)$ is pretty obvious, \nagyrev{since $q\longmapsto p_q^\text{unif}(a)$ is upper semicontouous
\eqref{p=usc} and $q\longmapsto p^\mathbf{E}_q(a)$ is lower semicontinuous \eqref{p=lsc}}.
The converse implication $(i)\Rightarrow (ii)$ follows immediately from Lemma \ref{usc-uniqueness}.
\end{proof}
 
\begin{definition}\label{def-proper}
A point $q_0\in Q$, satisfying the equivalent conditions in the preceding Proposition, will be referred to as a
 \nagyrev{\em point of continuous reduction relative to $\mathbf{E}$}. We denote by 
$Q^{\mathbf{E}}_{\text{r-cont}}$ the set of all such points.
On the one hand, using \eqref{pEleqpu}, it is pretty obvious that one has the inclusion
$Q^{\text{unif}}_{\text{simple}}\subset Q^{\mathbf{E}}_{\text{r-cont}}$. On the other hand,
besides the continuity property (i) above,
the maps \eqref{punif-map} are also continuous at points in $Q^\mathbf{E}_\text{r-cont}$.
\end{definition}

\begin{theorem}\label{r-cont=denseGdelta}
Suppose $C_0(Q)\subset A$ is a non-degenerate central inclusion, and 
$\mathbf{E}:A\to C_0(Q)$ is a conditional expectation.
If $A$ is separable and $\mathbf{E}$ is essentially faithful, then $Q^{\mathbf{E}}_{\text{\rm r-cont}}$
is a dense $G_\delta$-subset of $Q$. 
\end{theorem}

\begin{proof}
%
Fix some dense sequence $\{a_n\}_{n\in\mathbb{N}}$ in $A^+$ (the set of positive elements in $A$), which also satisfies 
\begin{equation}
p^{\mathbf{E}}_q(a_n)>0 ,\,\,\,\forall\,q\in Q,\,n\in\mathbb{N}.
\end{equation}
(This can be achieved with the help of a strictly positive function $f\in C_0(Q)$ and replacing, if necessary, $a_n$ with $a_n+\frac{1}nf$. Since $f$ is strictly positive, for each $q\in Q$, it follows that $\Gamma_{\text{ev}_q\circ\mathbf{E}}(f)$ is an invertible positive operator, namely $f(q)I$.)

Define, for every $r\in (0,1)\cap \mathbb{Q}$ and every $n\in\mathbb{N}$, the set
$$F_{n,r}=\{q\in Q\,:\,p^{\mathbf{E}}(a_n)\leq r \cdot p^{\text{unif}}(a_n)\}.$$
By \eqref{pEleqpu} and by density, it is obvious that
$$Q^{\mathbf{E}}_{\text{r-cont}}=Q\smallsetminus \bigcup_{r,n}F_{n,r},$$
so the desired conclusion will follow from Baire's Theorem, once we establish the following:

\begin{claim}
For each $r\in (0,1)\cap \mathbb{Q}$ and $n\in\mathbb{N}$, the set $F_{n,r}$ is closed and has empty interior.
\end{claim} 

The fact that $F_{n,r}$ is closed, is quite clear, since we can write its complement as
\begin{align*}
Q\smallsetminus F_{n,r}&=
\{q\in Q\,:\,p^{\mathbf{E}}_q(a_n)> r\cdot p^{\text{unif}}_q(a_n)\}=\\
&=\bigcup_{t\in (0,\infty)}\underbrace{\{q\in Q\,:\,p^{\mathbf{E}}_q(a_n)>rt\}}_{V_t}\cap
\underbrace{\{q\in Q\,:\,p^{\text{unif}}_q(a_n)<t\}}_{W_t},
\end{align*}
with $V_t$'s open by lower semicontinuity of \eqref{pE-map}, and the $W_t$'s open
by upper semicontinuity of \eqref{punif-map}.

To prove that $F_{n,r}$ has empty interior, argue by contradiction, assuming the existence of some 
non-identically zero non-negative continuous function $h\in C_c(Q)$ with $\text{supp}\,h\subset \text{Int}(F_{n,r})$, which gives us a positive element $0\neq a=ha_n\in A$, such that
\begin{itemize}
\item[$(a)$] $p^{\mathbf{E}}_q(a)=h(q)p^{\mathbf{E}}(a_n)\leq r\cdot h(q)p^{\text{unif}}(a_n)=
r\cdot p^{\text{unif}}_q(a)$, $\forall\,q\in \text{Int}(F_{n,r})$,
\item[$(b)$] $p^{\mathbf{E}}_q(a)=p^{\text{unif}}_q(a)=0$, $\forall\,q\in Q\smallsetminus \text{Int}(F_{n,r})$,
\end{itemize}
thus yielding
\begin{equation}\label{norm<r}
0\neq \|a\|=\sup_{q\in Q}p^{\mathbf{E}}_q(a)\leq r\cdot\sup_{q\in Q}p^{\text{unif}}_q(a)=r\cdot \|a\|,
\end{equation}
which is clearly impossible. (The first equality in \eqref{norm<r} follows from the essential faithfulness of
$\mathbf{E}$.) 
\end{proof}

\section{Application to \`etale group bundles}\label{bundles-sec}

The main example of groupoid $C^*$-algebras that exhibit central inclusions correspond to the so-called
{\em \'etale group bundles\/} (perhaps more suitably referred to as {\em isotropic groupoids\/}), which are those
\'etale groupoids $\mathcal{G}$ on which the range and source maps $r$ and $s$ coincide.
Of course, for such groupoids, we have
$u\mathcal{G}u=u\mathcal{G}=\mathcal{G}u$, $\forall\,u\in\mathcal{G}^{(0)}$, so (by our Convention set forth in 
Section \ref{grp-back}) all
$u\mathcal{G}$'s are countable discrete groups. 

\begin{nagyrevenv}
The discussion below, although limited to \'etale group bundles, is only meant to prepare the reader for Section \ref{grp-sec}, where we will apply our findings to the {\em interior isotropy\/} of arbitary \'etale groupoids.
\end{nagyrevenv}

\begin{convention}
In this (and only in this) section, $\mathcal{G}$ will be assumed to be an \`etale group bundle.
\end{convention}

\begin{remark}
Both inclusions $C_0(\mathcal{G}^{(0)})\subset C^*(\mathcal{G})$ and
$ C_0(\mathcal{G}^{(0)})\subset C_{\text{red}}^*(\mathcal{G})$
are non-degenerate and central.
\end{remark}
\begin{notations}
Fix some unit $u\in\mathcal{G}^{(0)}$. 
We denote 
by $J^{\text{full}}_u\lhd C^*(\mathcal{G})$ the closed two-sided ideal
in $C^*(\mathcal{G})$ generated by $C_{0,u}(\mathcal{G}^{(0)})=\{f\in C_0(\mathcal{G}^{(0)})\,:\,f(u)=0\}$.
(In Section \ref{central-sec}, this ideal was denoted by $J^{\text{unif}}_u$; we use this modified notation in order to distinguish between the two possible central inclusions: $C_0(\mathcal{G}^{(0)})\subset C^*(\mathcal{G})$ and
$C_0(\mathcal{G}^{(0)})\subset C_{\text{red}}^*(\mathcal{G})$.) The ideal
$\pi_{\text{red}}(J^{\text{full}}_u)\lhd C_{\text{red}}^*(\mathcal{G})$ will be denoted by 
$J^{\text{other}}_u$. Equivalently, 
$J^{\text{other}}_u$ is the closed two-sided ideal in $C_{\text{red}}^*(\mathcal{G})$ generated by 
$C_{0,u}(\mathcal{G}^{(0)})$.
The corresponding $C^*$-seminorms on $C^*(\mathcal{G})$ and $C_{\text{red}}^*(\mathcal{G})$ will be denoted by
$p^{\text{full}}_u$ and $p^{\text{other}}_u$, respectively.

As we have a conditional expectation $\mathbb{E}_{\text{red}}:C_{\text{red}}^*(\mathcal{G})
\to C_0(\mathcal{G}^{(0)})$, we also have (as defined in Section \ref{central-sec}) a $C^*$-seminorm
$p^{\mathbb{E}_{\text{red}}}_u$ on $C_{\text{red}}^*(\mathcal{G})$, defined by
$p^{\mathbb{E}_{\text{red}}}_u(a)=\|\Gamma_{\text{ev}_u\circ\mathbb{E}_{\text{red}}}(a)\|$.
On the full $C^*$-algebra, the corresponding $C^*$-seminorm is
$p^{\mathbb{E}}_u=p^{\mathbb{E}_{\text{red}}}_u\circ\pi_{\text{red}}$.
In particular, this gives rise to a canonical $*$-isomorphism 
$C^*(\mathcal{G})/\text{ker}\,p^{\mathbb{E}}_u\simeq C_{\text{red}}^*(\mathcal{G})/\text{ker}\,p^{\mathbb{E}_{\text{red}}}_u$.
\end{notations}

\begin{remark}
For each unit $u\in\mathcal{G}^{(0)}$, the quotient
$C^*(\mathcal{G})/J^{\text{full}}_u$ is canonically identified with
$C^*(\mathcal{G}u)$ -- the full group $C^*$-algebra of $\mathcal{G}u$.
This identification is constructed in two steps. First, we choose, for any
$\gamma\in \mathcal{G}u$, some open bisection $\gamma\in\mathcal{B}_\gamma\subset \mathcal{G}$ and
some $n_\gamma\in C_c(\mathcal{B}_\gamma)$, with $n_\gamma(\gamma)=1$, and let $\mathbf{w}_\gamma=
n_\gamma(\text{mod }J^{\text{full}}_u)\in C^*(\mathcal{G})/J^{\text{full}}_u$.
(If $(n',\mathcal{B}')$ is another such pair -- i.e. $\gamma\in\mathcal{B}'\subset\mathcal{G}$ is some bisection, and
$n'\in C_c(\mathcal{B}')$ is some function with $n'(\gamma)=1$, then
$n'-n_\gamma\in J^{\text{full}}_u$, thus $\mathbf{w}_\gamma$ does not depend on the choice of $\mathcal{B}_\gamma$ and $n_\gamma$.)
It is quite routine to check that $C^*(\mathcal{G})/J^{\text{full}}_u$ is unital, with unit
$\mathbf{w}_e$ (\nagyrev{with $e$ the neutral element in the group $\mathcal{G}u$}), and furthermore, the map 
\begin{equation}
\mathcal{G}u\ni\gamma\longmapsto\mathbf{w}_\gamma\in
\mathcal{U}(C^*(\mathcal{G})/J^{\text{full}}_u)
\label{Gu-to-quot}
\end{equation}
is a group homomorphism. Next, if $f\in C_c(\mathcal{G})$, then
$f-\sum_{\gamma\in \mathcal{G}u\cap\text{supp}\,f}f(\gamma)n_\gamma\in J^{\text{full}}_u$, so 
\eqref{Gu-to-quot} establishes a surjective $*$-homomorphism
$\Theta^{\text{full}}_u:C^*(\mathcal{G}u)\to C^*(\mathcal{G})/J^{\text{full}}_u$. Secondly, if we denote
by $(\mathbf{v}_\gamma)_{\gamma\in \mathcal{G}u}$ the standard unitary generators of
$C^*(\mathcal{G}u)$, then the map
\begin{equation}
C_c(\mathcal{G})\ni f\longmapsto \sum_{\gamma\in \mathcal{G}u\cap \text{supp}\,f}
f(\gamma)\mathbf{v}_\gamma\in C^*(\mathcal{G}u)
\label{rep-GGu}
\end{equation}
defines a 
$*$-representation of
$C_c(\mathcal{G})$, which yields a $*$-homomorphism 
$$\mathfrak{e}^{\text{full}}_u:C^*(\mathcal{G})\to
C^*(\mathcal{G}u).$$ 
\nagyrev{Since \eqref{rep-GGu} vanishes on all functions $h\in C_c(\mathcal{G}^{(0)})$ with $h(u)=0$, it follows that 
$\text{ker}\,\mathfrak{e}^\text{full}_u\supset J^\text{full}_u$, so
$\mathfrak{e}^{\text{full}}_u$ descends to a well defined $*$-homomorphism
$$\widehat{\mathfrak{e}^{\text{full}}_u}:C^*(\mathcal{G})/J^\text{full}_u \to
C^*(\mathcal{G}u).$$ 
The desired conclusion now follows from the pretty obvious observation 
that $\widehat{\mathfrak{e}^{\text{full}}_u}\circ\Theta^{\text{full}}_u=\text{Id}$, which implies that 
$\Theta^{\text{full}}_u$ is in fact a $*$-isomorphism, with inverse $\widehat{\mathfrak{e}^{\text{full}}_u}$; among other things, this also proves that $\text{ker}\,\mathfrak{e}^\text{full}_u= J^\text{full}_u$.}
\end{remark}

\begin{remark}
For each unit $u\in\mathcal{G}^{(0)}$, the quotients
$C^*(\mathcal{G})/\text{ker}\,p^{\mathbb{E}}_u\simeq C_{\text{red}}^*(\mathcal{G})/\text{ker}\,p^{\mathbb{E}_{\text{red}}}_u$
are canonically identified with
$C_{\text{red}}^*(\mathcal{G}u)$ -- the reduced group $C^*$-algebra of $\mathcal{G}u$.
First of all, since $\text{ev}_u\circ\mathbb{E}$ vanishes on $J^{\text{full}}_u$,  it factors through the quotient $*$-homomorphism $q_u:C^*(\mathcal{G})\to C^*(\mathcal{G})/J^{\text{full}}_u$, so we can write
$\text{ev}_u\circ\mathbb{E}=\phi_u\circ q_u$, for some state $\phi_u$ on $C^*(\mathcal{G})/J^{\text{full}}_u$.
The desired conclusion now follows from the observation that, if $\Theta^{\text{full}}_u:C^*(\mathcal{G}u)\to C^*(\mathcal{G})/J^{\text{full}}_u$ is the $*$-isomorphism defined in the preceding Remark, then $\phi_u\circ\Theta^{\text{full}}_u$ coincides with the canonical trace $\boldsymbol{\tau}_{\mathcal{G}u}$ on $C^*(\mathcal{G}u)$ (defined on the canonical generators by $\boldsymbol{\tau}_{\mathcal{G}u}(\mathbf{v}_\gamma)=\delta_{\nagyrev{e,\gamma}}$).
\end{remark}

\begin{mycomment}
To summarize the preceding two Remarks, if $u\in\mathcal{G}^{(0)}$ is some unit, we have two $*$-isomorphisms $\Theta^{\text{full}}_u$,
$\Theta^{\text{red}}_u$ and two quotient $*$-homomorphisms
$\theta_u$, $\omega_u$
$$
C^*(\mathcal{G}u)
\xrightarrow{\,\Theta^{\text{full}}_u\,}
C^*(\mathcal{G})/J^{\text{full}}_u
\xrightarrow{\,\theta_u\,}
C_{\text{red}}^*(\mathcal{G})/J^{\text{other}}_u
\xrightarrow{\,\omega_u\,}
C_{\text{red}}^*(\mathcal{G})/\text{ker}\,p^{\mathbb{E}_{\text{red}}}_u
\xleftarrow{\,\Theta^{\text{red}}_u\,}C^*_{\text{red}}(\mathcal{G}u),
$$
which yield a commutative diagram
$$
\xymatrix{C^*(\mathcal{G}) \ar[r]^{\pi_{\text{red}}} \ar[d] &C^*_{\text{red}}(\mathcal{G})\ar[d]\ar[dr]&\\
C^*(\mathcal{G})/J^{\text{full}}_u\ar[r]^{\theta_u} & C^*_{\text{red}}(\mathcal{G})/J^{\text{other}}_u
\ar[r]^{\omega_u\,\,\,\,\,}& C^*_{\text{red}}(\mathcal{G})/\text{ker}\,p^{\mathbb{E}_{\text{red}}}_u \\
C^*(\mathcal{G}u)\ar[u]^{\Theta^{\text{full}}_u}\ar[rr]&& C^*_{\text{red}}(\mathcal{G}u)
\ar[u]_{\Theta^{\text{red}}_u}
}
$$
(The bottom horizontal arrow designates the standard quotient given by the identification
$C^*_{\text{red}}(\mathcal{G}u)=C^*(\mathcal{G}u)/K_{\boldsymbol{\tau}_{\mathcal{G}u}}$.)
Following the second row, we can split the bottom rectangle into two commutative sub-diagrams
(which employ a third $*$-isomorphism $\Theta^{\text{other}}_u$)
  $$
\xymatrix{
C^*(\mathcal{G})/J^{\text{full}}_u\ar[r]^{\theta_u} & C^*_{\text{red}}(\mathcal{G})/J^{\text{other}}_u
\ar[r]^{\omega_u\,\,\,\,\,}& C^*_{\text{red}}(\mathcal{G})/\text{ker}\,p^{\mathbb{E}_{\text{red}}}_u \\
C^*(\mathcal{G}u)\ar[u]^{\Theta^{\text{full}}_u}\ar[r]^{\boldsymbol{\rho}_u}&
C^*_{\text{other},u}(\mathcal{G}u)\ar[r]^{\boldsymbol{\kappa}_u}\ar[u]^{\Theta^{\text{other}}_u}& 
C^*_{\text{red}}(\mathcal{G}u)\ar[u]_{\Theta^{\text{red}}_u}
}
$$
where the bottom arrows $\boldsymbol{\rho}_u$ and $\boldsymbol{\kappa}_u$ are surjective $*$-homomorphisms, and the intermediary $C^*$-algebra $C^*_{\text{other},u}(\mathcal{G}u)$ is understood as a quotient of
$C^*(\mathcal{G}u)/W_{\mathcal{G},u}$
the full group $C^*$-algebra
$C^*(\mathcal{G}u)$ by the ideal $W_{\mathcal{G},u}=(\Theta^{\text{full}}_u)^{-1}(\text{ker}\,\theta_u)$, which is 
contained in $K_{\boldsymbol{\tau}_{\mathcal{G}u}}$.

Equivalently, we have a commutative diagram
\begin{equation}
\xymatrix{
C^*(\mathcal{G})\ar[r]^{\pi_{\text{red}}} \ar[d]_{\mathfrak{e}^{\text{full}}_u}
&
C^*_{\text{red}}(\mathcal{G})\ar[d]_{\mathfrak{e}^{\text{other}}_u} \ar[dr]^{\mathfrak{e}^{\text{red}}_u}
&\\
C^*(\mathcal{G}u)\ar[r]_{\boldsymbol{\rho}_u}&
C^*_{\text{other},u}(\mathcal{G}u)\ar[r]_{\boldsymbol{\kappa}_u}& 
C^*_{\text{red}}(\mathcal{G}u)
}
\label{e-maps}
\end{equation}
involving three surjective $*$-homomorphisms
$\mathfrak{e}^{\text{full}}_u$,
$\mathfrak{e}^{\text{other}}_u$, and
$\mathfrak{e}^{\text{red}}_u$, having 
$\text{ker}\,\mathfrak{e}^{\text{full}}_u=J^{\text{full}}_u$, 
$\text{ker}\,\mathfrak{e}^{\text{other}}_u=J^{\text{other}}_u$,
$\text{ker}\,\mathfrak{e}^{\text{red}}_u=\text{ker}\,p^{\mathbb{E}_{\text{red}}}_u$.

The $C^*$-algebras $C^*(\mathcal{G}u)$, $C^*_{\text{other},u}(\mathcal{G}u)$ and $C^*_{\text{red}}(\mathcal{G}u)$ can all be understood as completions of the group $*$-algebra $\mathbb{C}[\mathcal{G}u]$ with respect to the $C^*$-norms $\|\,.\,\|_{\text{full}}$, $\|\,.\,\|_{\text{other},u}$ and $\|\,.\,\|_{\text{red}}$, which satisfy
$$\|\,.\,\|_{\text{full}}\geq\|\,.\,\|_{\text{other},u}\geq\|\,.\,\|_{\text{red}}.$$ 
We caution the reader that, unlike $\|\,.\,\|_{\text{full}}$ and $\|\,.\,\|_{\text{red}}$,
 the intermediary $C^*$-norm $\|\,.\,\|_{\text{other},u}$ is not intrinsically defined in terms of $\mathcal{G}u$ alone, as it depends on the particular way the group $\mathcal{G}u$ sits in $\mathcal{G}$.  

Furthermore (as seen in \cite[Prop 3]{HLS}), there are examples available in the literature in which the norms
$\|\,.\,\|_{\text{other},u}$ and $\|\,.\,\|_\text{red}$ may fail to coincide.
\end{mycomment}
 
\begin{remark}\label{G-isotropic-consequence}
As the expectation $\mathbb{E}_{\text{red}}:C^*_{\text{red}}(\mathcal{G})\to C_0(\mathcal{G}^{(0)})$ is essentially faithful (in fact honestly  faithful), using the terminology and the results from Section 
\ref{central-sec}, the inclusion $C_0(\mathcal{G}^{(0)})\subset C^*_{\text{red}}(\mathcal{G})$ exhibits the following properties
\begin{itemize}
\item[A.] The set of units \nagyrev{\em of continuous reduction relative to $\mathbb{E}_\text{red}$},  described (see Lemma \ref{usc-uniqueness}, Proposition \ref{cont-prop}, and Definition \ref{def-proper}) in five equivalent ways as
\begin{align*}
\mathcal{G}^{(0)}_{\text{r-ccont}}&=\{u\in\mathcal{G}^{(0)}\,:\,\lim_{v\to u}p^{\mathbb{E}_\text{red}}_v(f)=
p^{\mathbb{E}_\text{red}}_v(f),\,\,\,\forall\,f\in C_c(\mathcal{G})\}=\\
&=\{u\in\mathcal{G}^{(0)}\,:\,\limsup_{v\to u}p^{\mathbb{E}_\text{red}}_v(f)\leq
p^{\mathbb{E}_\text{red}}_v(f),\,\,\,\forall\,f\in C_c(\mathcal{G})\}=\\
&=\{u\in\mathcal{G}^{(0)}\,:\,J^{\text{red}}_u=\text{ker}\,p^{\mathbb{E}_{\text{red}}}_u\}=\{u\in\mathcal{G}^{(0)}\,:\,\boldsymbol{\kappa}_u\text{ is an isomorphism}\}=\\
&=\{u\in\mathcal{G}^{(0)}\,:\,\text{ the norms $\|\,.\,\|_{\text{other},u}$ and $\|\,.\,\|_{\text{red}}$
coincide on $\mathbb{C}[\mathcal{G}u]$}\}
\end{align*}
is a dense $G_\delta$ set in $\mathcal{G}^{(0)}$.
\item[B.] If $\mathcal{G}u$ is {\em amenable}, then the norms
$\|\,.\,\|_{\text{full}}$, $\|\,.\,\|_{\text{other},u}$  and $\|\,.\,\|_{\text{red}}$ all coincide on $\mathbb{C}[\mathcal{G}u]$, thus
$u\in\mathcal{G}^{(0)}_{\text{r-cont}}$.
\item[C.] If the set, defined in two equivalent ways as
\begin{align*}
\mathcal{X}_{\mathcal{G}}&=
\{u\in\mathcal{G}^{(0)}\,:\,C^*_{\text{other},u}(\mathcal{G}u) 
\text{ is a simple $C^*$-algebra}\}=\\
&=
\{u\in\mathcal{G}^{(0)}_{\text{r-cont}}\,:\,\mathcal{G}u \text{ is a {\em $C^*$-simple group\/} (i.e. $C^*_{\text{red}}(\mathcal{G}u)$ is simple)}\}
\end{align*}
is dense in $\mathcal{G}^{(0)}$, then the inclusion
$C_0(\mathcal{G}^{(0)})\subset C^*_{\text{red}}(\mathcal{G})$ is essential.
\end{itemize}
\end{remark}

In connection with Remark \ref{G-isotropic-consequence}.B, but also with an eye on Proposition
\ref{aug-ideal} 
below, we conclude this section with the following construction.

\begin{definition}
(Assume still that $\mathcal{G}$ is isotropic.) For any open bisection $\mathcal{B}\subset\mathcal{G}$, 
consider the isomorphism
$$\mathfrak{t}_{\mathcal{B}}:C_c(\mathcal{B})
\ni f\longmapsto f\circ (r|_{\mathcal{B}})^{-1}\in C_c(r(\mathcal{B})),$$
and let $\mathfrak{V}_c(\mathcal{G})$ denote the two-sided ideal 
in $C_c(\mathcal{G},\times)$
generated by
$$\bigcup_{\substack{\mathcal{B}\subset\mathcal{G}\\ \text{open bisection}}}
\{f-\mathfrak{t}_{\mathcal{B}}(f)\,:\,f\in C_c(\mathcal{B})\}\subset C_c(\mathcal{G}).$$
When we view $C_c(\mathcal{G})\subset C^*(\mathcal{G})$, we denote the norm-closure of
$\mathfrak{V}_c(\mathcal{G})$ by 
$\mathfrak{V}_{\text{full}}(\mathcal{G})$;
but when we view $C_c(\mathcal{G})\subset C^*_{\text{red}}(\mathcal{G})$, we denote the norm-closure of
$\mathfrak{V}_c(\mathcal{G})$ by 
$\mathfrak{V}_{\text{red}}(\mathcal{G})$. Each one of these ideals --
either in $C_c(\mathcal{G})$, or in $C^*(\mathcal{G})$, or in $C^*_{\text{red}}(\mathcal{G})$ -- will be referred to as the
{\em augmentation ideal}.
\end{definition}

\begin{proposition}\label{aug-ideal}
Assume $\mathcal{G}$ is isotropic.
\begin{itemize}
\item[(i)] The augmentation ideal $\mathfrak{V}_c(\mathcal{G})$  (and consequently
both $\mathfrak{V}_{\text{\rm full}}(\mathcal{G})$ and
$\mathfrak{V}_{\text{\rm red}}(\mathcal{G})$) is non-zero, if and only if $\mathcal{G}u_0\neq \{e\}$ for some
$u_0\in\mathcal{G}^{(0)}$; equivalently, $\mathcal{G}^{(0)}\subsetneq\mathcal{G}$.
\item[(ii)] One has a strict inclusion:
$\mathfrak{V}_{\text{\rm full}}(\mathcal{G})\subsetneq C^*(\mathcal{G})$.
\item[(iii)] If there exists $u_0\in\mathcal{G}^{(0)}$, for which the augmentation
homomorphism $\mathbb{C}[\mathcal{G}u_0]\to\mathbb{C}$ (defined by the trivial representation
of the group $\mathcal{G}u_0$) is continuous relative to the norm $\|\,.\,\|_{\text{\rm other},u_0}$, then
one also has a strict inclusion
$\mathfrak{V}_{\text{\rm red}}(\mathcal{G})\subsetneq C^*_{\text{\rm red}}(\mathcal{G})$.
\end{itemize}
\end{proposition}
(A sufficient condition for the hypothesis in statement (iii) to hold true is that $\mathcal{G}u_0$ is {\em amenable}, for some $u_0\in\mathcal{G}^{(0)}$.)

\begin{proof}
Statement (i) is pretty obvious. As for (ii) and (iii), a unified proof can be provided as follows.
Denote by $\mathfrak{V}$ either one of $\mathfrak{V}_{\text{full}}(\mathcal{G})$ or
$\mathfrak{V}_{\text{red}}(\mathcal{G})$, denote by $
\mathfrak{A}$ either one of $C^*(\mathcal{G})$ -- for statement (ii), or
$C^*_{\text{red}}(\mathcal{G})$ -- for statement (iii). Fix also some unit $u_0\in\mathcal{G}^{(0)}$, which either satisfies the hypothesis in (iii), or is arbitrary otherwise, denote by $\mathfrak{A}_{u_0}$ one of 
$C^*(\mathcal{G}u_0)$ 
-- for (ii), or $C^*_{\text{other},u_0}(\mathcal{G}u_0)$ -- for (iii), and lastly let $\mathfrak{e}_{u_0}:
\mathfrak{A}\to \mathfrak{A}_{u_0}$ denote one of the maps from \eqref{e-maps}. The trivial group representation
$\mathcal{G}u_0\ni\gamma\longmapsto 1\in\mathbb{T}$ induces an augmentation $*$-homomorphism
$\mathfrak{a}_{u_0}:\mathfrak{A}_{u_0}\to\mathbb{C}$. (In statement (ii) this is obvious; in statement (iii) this follows from the stated hypothesis.) Now the desired conclusion simply follows from the observation that,
for each open bisection $\mathcal{B}$, we have the inclusion
$$\{f-\mathfrak{t}_{\mathcal{B}}(f)\,:\,f\in C_c(\mathcal{B})\}\subset
\text{ker}(\mathfrak{a}_{u_0}\circ\mathfrak{e}_{u_0}),$$
which in turn implies $\mathfrak{V}\subset \text{ker}(\mathfrak{a}_{u_0}\circ\mathfrak{e}_{u_0})$.
\end{proof}

\section{Minimality and Simplicity}\label{min-sec}

In the context of (traditional) dynamical systems, an action $G\curvearrowright X$ of a group $G$ on a locally compact
space $X$ is said to be {\em minimal}, if there do not exists any non-trivial (i.e. non-empty strict subset) open $G$-invariant subsets of $X$. This notion is extended to arbitrary \'etale groupoids $\mathcal{G}$, by making a similar requirement on the unit space $\mathcal{G}^{(0)}$ (see section \ref{grp-sec} for more on this). Minimality is also extended to actions $G\curvearrowright B$  of a group $G$ by automorphisms of a $C^*$-algebra $B$, by requiring the non-existence of non-trivial $G$-invariant ideals $\{0\}\neq J \lhd B$. In what follows (see Definition \ref{def-minimal}.B) we propose yet another generalization of minimality, for non-degenerate of $C^*$-inclusions.

\begin{notation}
Suppose $B\subset A$ is a $C^*$-inclusion. An element $n\in A$ is called a {\em $B$-normalizer}, if
$nBn^*\cup n^*Bn\subset B$. The collection of all these normalizers is denoted by $\mathcal{N}_A(B)$.
\end{notation}

\begin{remark}\label{monoid}
It is fairly obvious that $\mathcal{N}_A(B)$ is norm-closed; furthermore it is also a \nagyrev{\em $*$-semigroup}, in the sense that
\begin{itemize}
\item[$(a)$] $n\in\mathcal{N}_A(B)\Rightarrow n^*\in\mathcal{N}_A(B)$;
\item[$(b)$] $n_1,n_2\in\mathcal{N}_A(B)\Rightarrow n_1n_2\in\mathcal{N}_A(B)$.
\end{itemize}
\end{remark}

\begin{remark} \label{norm-support}
In the case of a non-degenerate inclusion $B\subset A$, normalizers are supported in $B$, in the sense that:
 $nn^*,n^*n\in B$, $\forall\,n\in\mathcal{N}(B)$.
\end{remark}

\begin{definition}
Suppose  $B\subset A$ is a  non-degenerate $C^*$-inclusion.
Denote the $C^*$-subalgebra of $A$ generated by $\mathcal{N}_A(B)$ by 
$A^{\text{reg}}$. The resulting non-degenerate inclusion $B\subset A^{\text{reg}}$ is called
the {\em regular part\/} of $B\subset A$.
\end{definition}

\begin{mycomment}
The above terminology is consistent with the standard one, by which a non-degenerate inclusion $B\subset A$
is termed {\em regular}, if $A^{\text{reg}}=A$. Using this language, and the obvious equality
\begin{equation}
\mathcal{N}_{A^{\text{reg}}}(B)=\mathcal{N}_A(B).
\label{norm-reg}
\end{equation}
it follows that $A^{\text{reg}}$ is the largest $C^*$-subalgebra of $A$ that makes the inclusion 
$B\subset A^{\text{reg}}$ regular.

By Remark \ref{monoid} it is also fairly obvious that
\begin{equation}
A^{\text{reg}}=\overline{\text{span}}\mathcal{N}_A(B).
\label{Areg=span}
\end{equation}
\end{mycomment} 

In order to achieve more flexibility, it is helpful to consider normalizers sitting in multiplier algebras, the properties of which are collected in
Proposition \ref{norm-mult} below.

 
\begin{proposition}\label{norm-mult}
Assume $B\subset A$ is a non-degenerate $C^*$-inclusion, and let $M(B)\subset M(A)$ be the associated multiplier algebra inclusion (as described in Remark \ref{MBsubMA-rem}). 
\begin{itemize}
\item[(i)] $\mathcal{N}_{M(A)}(B)$ is a norm-closed \nagyrev{sub-$*$-monoid (i.e. a unital $*$-semigroup)} of $M(A)$; furthermore, if $n\in\mathcal{N}_{M(A)}(B)$ is invertible in $M(A)$, then $n^{-1}\in\mathcal{N}_{M(A)}(B)$.
\item[(ii)] $\mathcal{N}_A(B)=\mathcal{N}_{M(A)}(B)\cap A=B\mathcal{N}_{M(A)}(B)B=B\mathcal{N}_{A}(B)B$.
\item[(iii)] $\mathcal{N}_{M(A)}(M(B))=\mathcal{N}_{M(A)}(B)$; in particular,
$nn^*,n^*n\in M(B)$, $\forall\,n\in\mathcal{N}_{M(A)}(B)$.
\end{itemize}
\end{proposition}

\begin{proof}
Fix an approximate unit $(u_\lambda)_\lambda\subset B$ for $A$, and let us first prove (iii).
The inclusion $\mathcal{N}_{M(A)}(M(B))\subset\mathcal{N}_{M(A)}(B)$ is fairly obvious, 
by Remark \ref{MBsubMA-rem}, because whenever $n\in\mathcal{N}_{M(A)}(M(B))$ and $b\in B$, the elements
$nbn^*$ and $n^*bn$ clearly belong to both $M(B)$ and $A\unlhd M(A)$, so
$nbn^*,n^*bn\in M(B)\cap A=B.$
For the reverse inclusion $\mathcal{N}_{M(A)}(B)\subset\mathcal{N}_{M(A)}(M(B))$, fix for the moment some arbitrary elements $n\in\mathcal{N}_{M(A)}(B)$, $x\in M(B)$, and $b\in B$, and notice that,
since we have $nxn^*b=\lim_\lambda nxu_\lambda (n^*b)$ (with the element in parenthesis representing an element in $A$),
as well as $xu_\lambda\in M(B)B \subset B$, it follows that 
$n(xu_\lambda) n^*b\in nBn^*b\subset Bb\subset B$, so taking the limit, it follows that $nxn^*b\in B$.
Since $b\in B$ was arbitrary, we can conclude that the multiplier $m=nxn^*\in M(A)$ satisfies $mB\subset B$, so by 
Remark \ref{MBsubMA-rem}, it follows that $nxn^*\in M(B)$, and likewise, $n^*xn\in M(B)$. 
Letting $x\in M(B)$ vary, this simply shows that $nM(B)n^*\cup n^*M(B)n\subset M(B)$, so $n\in
\mathcal{N}_{M(A)}(M(B))$. 

The second statement from (iii), as well as the first statement from (i) are now obvious, by Remark \ref{norm-support} applied to the unital (thus non-degenerate) inclusion $M(B)\subset M(A)$. As for the second statement from (i), assume
$n\in\mathcal{N}_{M(A)}(B)$ is invertible, and let us consider its polar decomposition
$n=u(n^*n)^{1/2}$, where $u$ is a unitary in $M(A)$, and $(n^*n)^{1/2}$ is a positive invertible element in $M(B)$. Since
$M(B)\subset \mathcal{N}_{M(A)}(M(B))=\mathcal{N}_{M(A)}(B)$, it follows that $(n^*n)^{-1/2}\in \mathcal{N}_{M(A)}(B)$, so by the \nagyrev{monoid} property, the unitary $u=n(n^*n)^{-1/2}$ also belongs to $\mathcal{N}_{M(A)}(B)$, and then
again by the $*$-semigroup property $n^{-1}=(n^*n)^{-1/2}u^*$ also belongs to
$\mathcal{N}_{M(A)}(B)$.

(ii) First of all, the equality $\mathcal{N}_A(B)=\mathcal{N}_{M(A)}(B)\cap A$ is trivial from the definition.
Secondly, the inclusion $B\mathcal{N}_{M(A)}(B)B\subset \mathcal{N}_A(B)$ is fairly obvious, 
from the inclusion by the $B\mathcal{N}_{M(A)}(B)B\subset\mathcal{N}_{M(A)}(B)$ (which follows from the semigroup
property, as $B\subset M(B)\subset\mathcal{N}_{M(A)}(B)$), which combined with the observation that
$BM(A)B\subset AM(A)A\subset A$, implies that $B\mathcal{N}_{M(A)}(B)B\subset \mathcal{N}_{M(A)}(B)\cap A=\mathcal{N}_A(B)$.

So far, we have $\mathcal{N}_A(B)=\mathcal{N}_{M(A)}(M(B))\cap A\supset B\mathcal{N}_{M(A)}(M(B))B\supset B\mathcal{N}_{A}(B)B$, so all that remains to be justified is the factorization inclusion
$\mathcal{N}_A(B)\subset B\mathcal{N}_{A}(B)B$, for which (by applying adjoints) it suffices to prove the inclusion
$\mathcal{N}_A(B)\subset B\mathcal{N}_{A}(B)$. Fix $n\in \mathcal{N}_A(B)$, and consider the elements
$x_\varepsilon=\left((nn^*)^{1/4}+\varepsilon 1\right)^{-1}\in M(B)\subset\mathcal{N}_{M(A)}(B)$, $\varepsilon>0$, as well as the products
$y_\varepsilon =x_\varepsilon n\in A\cap\mathcal{N}_{M(A)}(B)=\mathcal{N}_A(B)$.
A simple calculation shows that
$$\|y_\varepsilon-y_\delta\|^2=\|n^*\left(\left((nn^*)^{1/4}+\varepsilon 1\right)^{-1}-
\left((nn^*)^{1/4}+\delta 1\right)^{-1}\right)^2n\|=
\|f_{\varepsilon.\delta}(n^*n)\|,$$
where $f_{\varepsilon,\delta}(t)=t\left(\dfrac{1}{\root{4}\of{t}+\varepsilon}-
\dfrac{1}{\root{4}\of{t}+\delta}\right)^2=(\varepsilon-\delta)^2\dfrac{t}{(\root{4}\of{t}+\varepsilon)^2
(\root{4}\of{t}+\delta)^2}$. Since $|f_{\varepsilon,\delta}(t)|\leq \nagyrev{(\varepsilon-\delta)^2}$, $\forall\,t\geq 0$, it follows that
$\|y_\varepsilon-y_\delta\|\leq|\varepsilon-\delta|$, so the limit  $y=\lim_{\varepsilon\to 0}y_\varepsilon$ exists in
$A$, so it defines an element $y\in\mathcal{N}_A(B)$. By construction, we have
$n=\left((nn^*)^{1/4}+\varepsilon 1\right)y_\varepsilon$, so taking limit we get
$n=(nn^*)^{1/4}y\in B\mathcal{N}_A(B)$.
\end{proof}

We also have the following multiplier version of \eqref{norm-reg}
 
\begin{proposition}\label{norm-mult-reg}
Suppose $B\subset A$ is a non-degenerate inclusion, let $B\subset A^{\text{\rm reg}}$ be its regular part. When one considers the associated multiplier inclusions $M(B)\subset M(A^{\text{\rm reg}})\subset M(A)$ 
(as described in Remark \ref{MBsubMA-rem}), the following equality holds:
\begin{equation}
\mathcal{N}_{M(A^{\text{\rm reg}})}(B)=
\mathcal{N}_{M(A)}(B).\label{norm-mult-reg-eq}
\end{equation}
\end{proposition}
\begin{proof}
The inclusion ``$\subset$'' being trivial, we only need to check the reverse inclusion
``$\supset$,'' which amounts to proving the inclusion
$\mathcal{N}_{M(A)}(B)\subset M(A^{\text{reg}})$, or equivalently, to showing that
$\mathcal{N}_{M(A)}(B)\cdot A^{\text{reg}}\subset A^{\text{reg}}$. By \nagyrev{\eqref{Areg=span}},
it suffices to prove the inclusion
$\mathcal{N}_{M(A)}(B)\cdot\mathcal{N}_A(B)\subset \mathcal{N}_A(B)$, which is immediate from
Proposition \ref{norm-mult} and Remark \ref{monoid}.
\end{proof}

\begin{definitions}\label{def-minimal}
Suppose $B\subset A$ is a non-degenerate $C^*$-inclusion.
\begin{itemize}
\item[A.] An ideal $J\unlhd B$ is said to be {\em fully normalized in $A$, relative to $B$}, 
if
$\mathcal{N}_{M(A)}(B)\subset \mathcal{N}_{M(A)}(J)$.
\item[B.] 
We declare the inclusion $B\subset A$ {\em minimal}, if the only ideals $J\unlhd B$, that are fully normalized
in $A$ relative to $B$, 
are the trivial ones: $J=\{0\}$ and $J=B$.
\end{itemize}
\end{definitions}

\begin{nagyrevenv}
\begin{mycomment}
As we shall see in Section \ref{grp-sec}, when applied to groupoid $C^*$-algebras, the minimality of inclusions of the 
form $C_0(\mathcal{G}^{(0)}\subset C^*_\text{red}(\mathcal{G})$ corresponds precisely to the minimality of the gropupoid $\mathcal{G}$.
\end{mycomment}
\end{nagyrevenv} 

\begin{remark}
For an ideal $J\unlhd B$, the condition $\mathcal{N}_{M(A)}(B)\subset \mathcal{N}_{M(A)}(J)$, that appears in the above Definition,  is equivalent to the inclusion $\mathcal{N}_A(B)\subset\mathcal{N}_A(J)$. Indeed, if
$\mathcal{N}_{M(A)}(B)\subset \mathcal{N}_{M(A)}(J)$, then $\mathcal{N}_A(B)=A\cap\mathcal{N}_{M(A)}(B)
\subset A\cap\mathcal{N}_{M(A)}(J)=\mathcal{N}_A(J)$. Conversely, if
$\mathcal{N}_{A}(B)\subset \mathcal{N}_{A}(J)$, and $m\in\mathcal{N}_{M(A)}(B)$, then for every $x\in J$, we can write $mxm^*=\lim_{\lambda,\mu}(mu_\lambda)x(u_\mu m^*)$, for some approximate unit
$(u_\lambda)_\lambda$ for $B$, with both monomials in parentheses being elements in
$B\mathcal{N}_{M(A)}(B)\subset\mathcal{N}_A(B)\subset \mathcal{N}_A(J)$, thus proving that
$mxm^*\in J$ (as well as $m^*xm\in J$, by same argument applied to $m^*$); this argument shows that
$\mathcal{N}_{M(A)}(B)\subset\mathcal{N}_{M(A)}(J)$.
\end{remark}

\begin{remark}\label{A-inv-example}
For any ideal $L\unlhd A$, the intersection $B\cap L\unlhd B$ is fully normalized in $A$, relative to $B$.
Indeed, for any $n\in\mathcal{N}_A(B)$ and any $x\in B\cap L$, the products $nxn^*$ and $n^*xn$ both belong to $B$ (since $n$ normalizes $B$) as well as to $L$ (which is a two-sided ideal), thus proving the
inclusion $\mathcal{N}_A(B)\subset \mathcal{N}_A(B\cap L)$.
\end{remark}

When verifying (non-)minimality, the following result is particularly useful. \nagyrev{(See for instance Remark \ref{rem-nfn*}.)}

\begin{lemma}\label{normalized-ideal}
Assume $B\subset A$ is a non-degenerate $C^*$-inclusion. If an ideal $J\lhd B$ satisfies
the condition
\begin{itemize}
\item[{\sc (sn)}] $\mathcal{N}_A(B)\cap \mathcal{N}_A(J)$ generates $A^{\text{\rm reg}}$ as a $C^*$-subalgebra of $A$,
\end{itemize}
then $L=\overline{\text{\rm span}}\{a_1xa_2\,:\,x\in J,\,a_1,a_2\in A^{\text{\rm reg}}\}$ 
is a proper ideal $L\lhd A^{\text{\rm reg}}$.

In particular (by the preceding Remark), any ideal $J\lhd B$ satisfying condition {\sc (sn)} is contained in some ideal $J'\lhd B$, which is fully normalized in $A$ relative to $B$.
\end{lemma}
An ideal satisfying condition {\sc (sn)} is called {\em sufficiently normalized in  $A$, relative to $B$}.

\begin{proof} 
By construction, $L$ is clearly a closed two-sided ideal in $A^{\text{reg}}$, which contains $J$, so it suffices to prove that $L\subsetneq A^{\text{reg}}$. (By non-degeneracy, this will also ensure the strict inclusion $B\cap L\subsetneq B$.)

Fix a state $\varphi\in S(B)$ which vanishes on $J$, let $\psi\in S(A^{\text{reg}})$ be an extension of
$\varphi$ to a state on $A^{\text{reg}}$, and let $K_\psi\lhd A^{\text{reg}}$ be the kernel of the GNS representation
$\Gamma_\psi:A^{\text{reg}}\to\mathscr{B}(L^2(A^{\text{reg}},\psi))$. Since $K_\psi$ is a proper ideal, the desired conclusion will follow if we prove the inclusion $J\subset K_\psi$ (because this will also force $L\subset K_\psi$). In other words, all we need is to prove that
$\Gamma_\psi$ vanishes on $J$, which amounts to showing that
\begin{equation}
\|xa\|_{2,\psi}=0,\,\,\,\forall\,x\in J,\,a\in A^{\text{reg}}.
\label{GNSx=0}
\end{equation}
By Remark \ref{monoid}, combined with condition {\sc (sn)}, we know that
$A^{\text{reg}}=\overline{\text{span}}\left[\mathcal{N}_A(B)\cap\mathcal{N}_A(J)\right]$, so it suffices to verify  
\eqref{GNSx=0} only for $a\in\mathcal{N}_A(B)\cap\mathcal{N}_A(J)$, which in turn is equivalent to checking
that
\begin{equation}
\psi(a^*x^*xa)=0,\,\,\,\forall\,x\in J,\,a\in \mathcal{N}_A(B)\cap \mathcal{N}_A(J).
\label{GNSx=0alt}
\end{equation}
However, the above equalities are trivial, since $\psi\big|_J=\varphi\big|_J=0$ and
$a^*x^*xa\in a^*Ja\subset J$, $\forall\,x\in J$, $a\in \mathcal{N}_A(J)$.
\end{proof}

\begin{proposition}\label{reg-simple}
Suppose $B\subset A$ is a non-degenerate $C^*$-inclusion, which is regular. Then the following conditions are equivalent.
\begin{itemize}
\item[(i)] The ambient $C^*$-algebra $A$ is simple.
\item[(ii)] The inclusion $B\subset A$ is essential and minimal.
\end{itemize}
\end{proposition}

\begin{proof}
$(i)\Rightarrow (ii)$. Assume $A$ is simple. The fact that $B\subset A$ is essential is obvious. Minimality is also fairly obvious, for if $J\lhd B$ is fully normalized in $A$, relative to $B$, then by Lemma
\ref{normalized-ideal}, there is some (proper) ideal $L\lhd A$, such that $J\subset L$, which by simplicity
forces $L=\{0\}$, thus $J=\{0\}$ as well.

$(ii)\Rightarrow (i)$. Assume now $B\subset A$ is essential and minimal, and let us prove that $A$ is simple.
Fix some proper ideal $L\lhd A$, and let us justify that $L$ must be equal to $\{0\}$. Since $B\subset A$ is non-degenerate, the intersection $B\cap L$ is a proper ideal in $B$.  As $B\subset A$ is essential, it suffices to prove that
$B\cap L=\{0\}$. However, this is immediate from minimality, since by Remark \ref{A-inv-example}, $B\cap L\lhd B$ is fully normalized in $A$ relative to $B$
\end{proof}

\begin{example}
Suppose $\alpha:G\ni g\longmapsto \alpha_g\in\text{Aut}(B)$ is a an action of a discrete group $G$ by automorphisms on a $C^*$-algebra $B$. The {\em full\/} (resp. {\em reduced\/}) crossed product $C^*$-algebra constructions
yield two non-degenerate regular inclusions
$B\subset B\ltimes_\alpha G$ and
$B\subset B\ltimes_{\alpha,\text{red}} G$. If $A$ denotes either one of these $C^*$-algebras, we always have a group homomorphism $G\ni g\longmapsto u_g\in\mathcal{U}(M(A))$, and a canonical subset
$\mathscr{X}=\{bu_g\,:\,b\in B,\,g\in G\}\subset \mathcal{N}_A(B)$, which generates $A$ as a $C^*$-algebra.
Within this framework, the inclusion $B\subset A$ is minimal (in the sense of Definition 
\ref{def-minimal}.B), if and only if $\alpha$ is a {\em minimal action}, in the sense that
$B$ contains no non-zero ideal
$J\lhd B$ such that $\alpha_g(J)=J$, $\forall\,g\in G$.

As Archbold and Spielberg showed (\cite{AS}), a sufficient condition for making the inclusion
$B\subset B\ltimes_{\alpha,\text{red}} G$ essential is that the action $\alpha$ is {\em topologically free}.
\end{example}

\begin{nagyrevenv}
The technical result below will be used in the next section.
\begin{lemma}\label{central-norm-ideal}
Assume we have two non-degenerate inclusions $C_0(Q)\subset D\subset A$, with $C_0(Q)\subset D$ central, along with two points $q_1,q_2\in Q$. For an element $n\in\mathcal{N}_A(C_0(Q))\cap\mathcal{N}_A(D)$, the following conditions are equivalent:
\begin{itemize}
\item[(i)] $nC_{0,q_1}(Q)n^*\subset C_{0,q_2}(Q)$;
\item[(ii)] $nJ^\text{\rm unif}_{q_1}n^*\subset J^\text{\rm unif}_{q_2}$.
\end{itemize}
{\rm (The ideals that appear in (ii) are the ideals in $D$, associated with the central inclusion $C_0(Q)\subset D$; in other words, $J^\text{unif}_{q_j}$ is the closed two-sided ideal in $D$, generated by $C_{0,q_j}(Q)$, $j=1,2.$)} 
\end{lemma}
\begin{proof}
The implication $(ii)\Rightarrow (i)$ is fairly obvious, since for every $q\in Q$ one has the equality $C_0(Q)\cap J^\text{unif}_q=C_{0,q}(Q)$.

For the other implication, assume (i) and note that by continuity of the map
$x\longmapsto nxn^*$, all we need to prove is that, for any $f\in C_{0,q_1}(Q)$ and any $d\in D$, the element
$nfdn^*$ belongs to $J^\text{unif}_{q_2}$. For this purpose, we fix a sequence $(h_k)_k$ of polynomials, such that
$\lim_{k\to\infty}t^2h_k(t)=t$, uniformly on $[0,\|n\|^2]$. (Approximate the functions $\root{k}\of{t}$ with polynomials without constant term.) With this choice in mind, it follows that
$n=\lim_knh_k(n^*n)n^*n$, so we can write
\begin{equation}
nfdn^*=\lim_k nh_k(n^*n)n^*nfdn^*=
\lim_k (nh_k(n^*n)fn^*)(ndn^*).
\label{central-norm-ideal-eq}
\end{equation}
Since $h_k(n^*n)\in C^b(Q)$, by our assumption on $f$, we know that the elements $h_k(n^*n)f$ belong to $C_{0,q_1}(Q)$, so the elements $nh_k(n^*n)fn^*$ belong to $nC_{0,q_1}(Q)n^*\subset C_{0,q_2}(Q)$, by condition (i). Now we are done, because $ndn^*$ belongs to
$D$, so \eqref{central-norm-ideal-eq} forces $nfdn^*\in \overline{C_{0,q_2}(Q)D}=J^\text{unif}_{q_2}$.
\end{proof}
\end{nagyrevenv}

\section{Applications to Groupoid \texorpdfstring{$C^*$}{C*}-algebras}\label{grp-sec}

\begin{convention}
Besides the standard conditions set forth at the beginning of Section \ref{grp-back}, all groupoids mentioned in this section are also assumed to be \'etale.
\end{convention}

\begin{remark}
As pointed out in \cite{Renault3}, in the case of an \'etale groupoid $\mathcal{G}$,
the non-degenerate inclusions $C_0(\mathcal{G}^{(0)})\subset C^*(\mathcal{G})$ and
$C_0(\mathcal{G}^{(0)})\subset C_{\text{red}}^*(\mathcal{G})$ are regular, since a large supply of normalizers consists of functions supported on bisections. Specifically, if we consider the space
$$\mathfrak{N}(\mathcal{G})=\bigcup_{\substack{\mathcal{B}\subset\mathcal{G}\\ \text{open bisection}}}C_c(\mathcal{B})\subset C_c(\mathcal{G}),$$
then $C_c(\mathcal{G})=\text{span}\,\mathfrak{N}(\mathcal{G})$, and furthermore all functions in $\mathfrak{N}(\mathcal{G})$ (hereafter referred to as {\em elementary normalizers\/})
normalize $C_0(\mathcal{G}^{(0)})$ in both $C^*(\mathcal{G})$ and in $C^*_{\text{red}}(\mathcal{G})$.
(As above, we view $C_c(\mathcal{G})$ as a dense $*$-subalgebra in both $C^*(\mathcal{G})$ and $C_{\text{red}}^*(\mathcal{G})$.)
\begin{nagyrevenv}
This statement can be justified in two ways.
\begin{itemize}
\item[A.] For any open bisection $\mathcal{B}$, one has the inclusion
\begin{equation}
\mathcal{B}\mathcal{G}^{(0)}\mathcal{B}^{-1}\subset
\mathcal{G}^{(0)}.\label{BGB}
\end{equation}
so in particular, if
 $n\in C_c(\mathcal{B})$, then using \eqref{sup-prod} and \eqref{sup-inv},
it follows that, for every $f\in C_c(\mathcal{G}^{(0)})$ the function
we have
$n\times f\times n^*$ is again supported in
$\mathcal{G}^{(0)}$.
\item[B.] Whenever $\mathcal{B}_1,\mathcal{B}_2\subset\mathcal{G}$ are open bisections, and $n_j\in C_c(\mathcal{B}_j)$, $j=1,2$, it follows that, for every $f\in C_c(\mathcal{G})$, one has the equality
\begin{equation}
\begin{split}
&(n_1\times f\times n_2)(\gamma)=
\\
&=\left\{
\begin{array}{l}
n_1\big((r|_{\mathcal{B}_1})^{-1}(r(\gamma))\big)
n_2\big((s|_{\mathcal{B}_2})^{-1}(s(\gamma))\big)
f\big([(r|_{\mathcal{B}_1})^{-1}(r(\gamma))]^{-1}\gamma
[(s|_{\mathcal{B}_2})^{-1}(s(\gamma))]^{-1}\big)
,\\
\qquad\qquad \text{if }r(\gamma)\in r(\mathcal{B}_1)
\text{ and }s(\gamma)\in s(\mathcal{B}_2)\\
0\text{, otherwise}
\end{array}
\right.
\end{split}
\label{n1fn2}
\end{equation}
When we specialize \eqref{n1fn2} to the particular case when $f\in C_c(\mathcal{G}^{(0)})$, and $n\in C_c(\mathcal{B})$ -- for some open bisection
$\mathcal{B}\subset\mathcal{G}$, we explicitly obtain:
\begin{equation}
(n\times f\times n^*)(\gamma)=
\left\{
\begin{array}{l}
\left|n\big((r|_{\mathcal{B}})^{-1}(\gamma)\big)\right|^2f\left(s\big((r|_{\mathcal{B}})^{-1}(\gamma)\big)\right)
\text{, if }\gamma\in r(\mathcal{B})\\
0\text{, otherwise}
\end{array}\right.
\label{nfn*}
\end{equation}
so, $n\times f\times n^*$ is indeed supported in $\mathcal{G}^{(0)}$.
\end{itemize}
\end{nagyrevenv}
\end{remark} 

\begin{nagyrevenv}
\begin{remark}\label{rem-nEn}
For an open bisection $\mathcal{B}$, besides the inclusion \eqref{BGB}, we also have the inclusion
\begin{equation}
\mathcal{B}(\mathcal{G}\smallsetminus \mathcal{G}^{(0)})\mathcal{B}^{-1}\subset
\mathcal{G}\smallsetminus \mathcal{G}^{(0)}.\label{BGnotB}
\end{equation}
Using this inclusion it follows that conditional expectation $\mathbb{E}:C^*(\mathcal{G})
\to C_0(\mathcal{G}^{(0)})$ from Section \ref{grp-back} is {\em normalized by all $n\in\mathfrak{N}(\mathcal{G})$}, in the sense that: 
\begin{equation}
\mathbb{E}(nan^*)=n\mathbb{E}(a)n^*,\,\,\,\forall\,n\in\mathfrak{N}(\mathcal{G}), \,a\in C^*(\mathcal{G}).
\label{Enormalized}
\end{equation}
Indeed, if we start with some $f\in C_c(\mathcal{G})$, then $\mathbb{E}(f)-f$ is supported in $\mathcal{G} \smallsetminus \mathcal{G}^{(0)}$, which by \eqref{sup-prod}, \eqref{sup-inv} and \eqref{BGnotB} implies that
$n\times (\mathbb{E}(f)-f)\times n^*$ is again supported in $\mathcal{G} \smallsetminus \mathcal{G}^{(0)}$, thus it vanishes when restricted to $\mathcal{G}^{(0)}$, i.e.
$\mathbb{E} (n\times (\mathbb{E}(f)-f)\times n^*)=0$,
which means that
$$\mathbb{E}(n\times f\times n^*)=\mathbb{E}(n\times \mathbb{E}(f)\times n^*)=
n\times \mathbb{E}(f)\times n^*.$$
(For the last equality, we use the fact that $n$ normalizes $C_0(\mathcal{G}^{(0)})$, so
the element $n\times \mathbb{E}(f)\times n^*$ belongs to $C_0(\mathcal{G}^{(0)})$, thus it is fixed by $\mathbb{E}$.)
\end{remark}
\end{nagyrevenv}

\begin{remark}\label{rem-nfn*}
The formula
\eqref{nfn*} is also valid for $f\in C_0(\mathcal{G}^{(0)})$. If $\mathfrak{A}$ denotes any one of $C^*(\mathcal{G})$ or
$C_{\text{red}}^*(\mathcal{G})$, using Lemma \ref{normalized-ideal}, the following conditions are equivalent.
\begin{itemize}
\item[(i)] The inclusion $C_0(\mathcal{G}^{(0)})\subset \mathfrak{A}$ is minimal, in the sense of Definition \ref{def-minimal}.B.
\item[(ii)] The only ideals $J\unlhd C_0(\mathcal{G}^{(0)})$ satisfying
$nJn^*\subset J,\,\,\,\forall\,n\in\mathfrak{N}(\mathcal{G}),$
are the trivial ideals: $J=\{0\}$ and $J=C_0(\mathcal{G}^{(0)})$.
\item[(iii)] Every non-empty subset $\mathcal{X}\subset\mathcal{G}^{(0)}$ satisfying the inclusion
\begin{equation}
r\big(s^{-1}(\mathcal{X})\big)\subset\mathcal{X}
\label{def-Xinv}
\end{equation}
is dense in $\mathcal{G}^{(0)}$.
\item[(iii')] The only subsets $\mathcal{X}\subset\mathcal{G}^{(0)}$, which satisfy \eqref{def-Xinv} and are either open or closed, are the trivial ones: $\mathcal{X}=\varnothing$ and
$\mathcal{X}=\mathcal{G}^{(0)}$.
\end{itemize}
According to the standard groupoid terminology, these conditions characterize the so-called {\em minimal\/} groupoids.
(Whether $\mathcal{G}$ is minimal or not, a set of units satisfying \eqref{def-Xinv} is termed {\em invariant}.)
\end{remark}

\begin{mycomment}
By Proposition \ref{reg-simple}, minimality of $\mathcal{G}$ is always a necessary condition for the simplicity of
either $C^*(\mathcal{G})$ or
$C^*_\text{red}(\mathcal{G})$. 
Although it is tempting to try to fit groupoid minimality in the framework provided in 
Proposition \ref{reg-simple}, this approach will be less fruitful, especially in the reduced case.
\end{mycomment}

\begin{notations}
Given an \'etale groupoid $\mathcal{G}$, we denote  its {\em isotropy groupoid}, i.e. the set 
$\{\gamma\in\mathcal{G}\,:\,r(\gamma)=s(\gamma)\}$ by $\text{Iso}(\mathcal{G})$.
Although $\text{Iso}(\mathcal{G})$ is a nice groupoid in its own right, its {\em interior}, hereafter denoted by 
$\text{IntIso}(\mathcal{G})$, is more manageable, because the natural inclusion
$C_c(\text{IntIso}(\mathcal{G}))\subset C_c(\mathcal{G})$  gives rise to a natural {\em $C^*$-inclusion} 
\begin{equation}
\mathfrak{i}_{\text{full}}:C^*(\text{IntIso}(\mathcal{G}))
\hookrightarrow C^*(\mathcal{G}).
\label{C*-iso-incl}
\end{equation}
(The existence of $\mathfrak{i}_{\text{full}}$ as a $*$-homomorphism is obvious, by the definition of 
$\|\,.\,\|_{\text{full}}$.
We  believe the injectivity of $\mathfrak{i}_{\text{full}}$ is well known, but in case it is not, we provide a proof in the Appendix.)

On the other hand, since $\text{IntIso}(\mathcal{G})$ and $\mathcal{G}$ have the same unit space
$\mathcal{G}^{(0)}$, we also have a commutative diagram of conditional expectations
$$
\xymatrix{
C^*(\text{IntIso}(\mathcal{G}))\ar[r]^{\quad\mathfrak{i}_{\text{full}}}\ar[d]_{\mathbb{E}^{\text{IntIso}(\mathcal{G})}}
& C^*(\mathcal{G})\ar[d]^{\mathbb{E}^{\mathcal{G}}}\\
C_0(\text{IntIso}(\mathcal{G})^{(0)})\ar@{=}[r] &C_0(\mathcal{G}^{(0)})
}
$$
which gives rise to another natural $C^*$-inclusion
\begin{equation}
\mathfrak{i}_{\text{red}}:C^*_{\text{red}}(\text{IntIso}(\mathcal{G}))\hookrightarrow C^*_{\text{red}}(\mathcal{G}).
\label{C*red-iso-incl}
\end{equation}
(The injectivity of $\mathfrak{i}_{\text{red}}$ can be justified, for instance, using Proposition \ref{embed-crit}.)
\end{notations}

\begin{mycomment}
We caution the reader that, for an arbitrary unit $u\in\mathcal{G}^{(0)}$, one might have a strict inclusion
of isotropy groups $\text{IntIso}(\mathcal{G})u\subsetneq\text{Iso}(\mathcal{G})u(=u\mathcal{G}u)$; nevertheless, in this case. $\text{IntIso}(\mathcal{G})u$ is always a {\em normal\/} subgroup of $\text{Iso}(\mathcal{G})u$.
However, as pointed out in \nagyrev{\cite[Lemma 3.3]{BNRSW}}, under our overall assumptions on $\mathcal{G}$, the set
$$\mathcal{G}^{(0)}_\circ=\{u\in\mathcal{G}^{(0)}\,:\,\text{IntIso}(\mathcal{G})u=\text{Iso}(\mathcal{G})u\}$$
is a dense $G_\delta$ subset of $\mathcal{G}^{(0)}$. 
\end{mycomment}

\begin{notations}
We adopt the same notations as those from Section \ref{bundles-sec}, for the \'etale group bundle
$\text{IntIso}(\mathcal{G})$ (with unit space $\mathcal{G}^{(0)}$), so now, for each unit $u\in\mathcal{G}^{(0)}$,
 the commutative diagrams
\eqref{e-maps} become
\begin{equation}
\xymatrix{
C^*(\text{IntIso}(\mathcal{G}))\ar[r]^{\pi^{\text{IntIso}(\mathcal{G})}_{\text{red}}} \ar[d]_{\mathfrak{e}^{\text{full}}_u}
&
C^*_{\text{red}}(\text{IntIso}(\mathcal{G}))\ar[d]_{\mathfrak{e}^{\text{other}}_u} \ar[dr]^{\mathfrak{e}^{\text{red}}_u}
&\\
C^*(\text{IntIso}(\mathcal{G})u)\ar[r]_{\boldsymbol{\rho}_u}&
C^*_{\text{other},u}(\text{IntIso}(\mathcal{G})u)\ar[r]_{\boldsymbol{\kappa}_u}& 
C^*_{\text{red}}(\text{IntIso}(\mathcal{G})u)
}
\label{iso-e-maps}
\end{equation}
The three surjective $*$-homomorphisms pointing downwards in the above diagram have the following kernels:
\begin{itemize}
\item $\text{ker}\,\mathfrak{e}^{\text{full}}_u=J^{\text{full}}_u$, the closed two-sided ideal in $C^*(\text{IntIso}(\mathcal{G}))$ generated by
$C_{0,u}(\mathcal{G}^{(0)})$;
\item $\text{ker}\,\mathfrak{e}^{\text{other}}_u=J^{\text{other}}_u$, the closed two-sided ideal  in $C^*_{\text{red}}(\text{IntIso}(\mathcal{G}))$ generated by
$C_{0,u}(\mathcal{G}^{(0)})$;
\item $\text{ker}\,\mathfrak{e}^{\text{red}}_u=\text{ker} \,p^{\mathbb{E}_{\text{red}}}_u=
K_{\text{ev}_u\circ\mathbb{E}_{\text{red}}}$, the kernel of the GNS representation associated with the state
$\text{ev}_u\circ\mathbb{E}_{\text{red}}=\boldsymbol{\tau}^{\text{red}}_{\text{IntIso}(\mathcal{G})u}\circ
\mathfrak{e}^{\text{red}}_u$ (here $\boldsymbol{\tau}^{\text{red}}_{\text{IntIso}(\mathcal{G})u}$ is the canonical trace
on the reduced group $C^*$-algebra $C^*_{\text{red}}(\text{IntIso}(\mathcal{G})u)$);
\end{itemize}
as in Section \ref{bundles-sec}, by construction, we have $\pi^{\text{IntIso}(\mathcal{G})}_{\text{red}}(J^{\text{full}}_u)=
J^{\text{other}}_u\subset\text{ker}\,\mathfrak{e}^{\text{red}}_u$.
By a slight abuse in notation, we will denote $\text{IntIso}(\mathcal{G})^{(0)}_{\text{r-cont}}$ simply by
$\mathcal{G}^{(0)}_{\text{r-cont}}$; in other words,
$$\mathcal{G}^{(0)}_{\text{r-cont}}=\{u\in\mathcal{G}^{(0)}\,:\,\text{ker}\,\mathfrak{e}^{\text{other}}_u=
\text{ker}\,\mathfrak{e}^{\text{red}}_u\}.$$
\end{notations}

\begin{remark}
\begin{nagyrevenv}
For any open bisection $\mathcal{B}$ we have the inclusion
$$\mathcal{B}\,\text{IntIso}(\mathcal{G})\mathcal{B}^{-1}\subset
\text{IntIso}(\mathcal{G}),$$
so using \eqref{sup-prod} and \eqref{sup-inv}, it follows that for every 
$n\in C_c(\mathcal{B})$ and every $f\in C_c(\text{IntIso}(\mathcal{G}))$, the function $n\times f\times n^*$ again belongs to
$C_c(\text{IntIso}(\mathcal{G}))$. Therefore we have the following properties. 
\end{nagyrevenv}
\begin{itemize}
\item[A.] When we view $\mathfrak{N}(\mathcal{G})$ as a subset of $C^*(\mathcal{G})$, and $C^*(\text{IntIso}(\mathcal{G}))$ as a $C^*$-subalgebra of $C^*(\mathcal{G})$, all elements in $\mathfrak{N}(\mathcal{G})$ normalize
$C^*(\text{IntIso}(\mathcal{G}))$, so the inclusion \eqref{C*-iso-incl} is regular.
\item[B.] When we view $\mathfrak{N}(\mathcal{G})$ as a subset of $C^*_{\text{red}}(\mathcal{G})$, and 
$C^*_{\text{red}}(\text{IntIso}(\mathcal{G}))$ as a $C^*$-subalgebra of $C^*_{\text{red}}(\mathcal{G})$, all elements in $\mathfrak{N}(\mathcal{G})$ normalize
$C^*_{\text{red}}(\text{IntIso}(\mathcal{G}))$, so the inclusion \eqref{C*red-iso-incl} is also regular.
\end{itemize}
\end{remark}

\begin{lemma}\label{lemma-eu}
Assume $\boldsymbol{\omega}=(\gamma,\mathcal{B},n)$ is a triple consisting of an element $\gamma\in\mathcal{G}$,
an open bisection $\mathcal{B}\ni \gamma$, and some function $n\in C_c(\mathcal{B})$ with $n(\gamma)=1$.
\begin{itemize}
\item[(i)] When viewing $C_c(\mathcal{B})\subset\mathfrak{N}(\mathcal{G})\subset
\mathcal{N}_{C^*(\mathcal{G})}\big(C^*(\text{\rm IntIso}(\mathcal{G}))\big)\subset
 C^*(\mathcal{G})$, the map
\begin{equation}
\text{\rm ad}^{\text{\rm full}}_n:C^*(\text{\rm IntIso}(\mathcal{G}))\ni a\longmapsto
nan^*\in  C^*(\text{\rm IntIso}(\mathcal{G}))
\label{ad-full}
\end{equation}
sends the ideal $J^{\text{\rm full}}_{s(\gamma)}=\text{\rm ker}\,\mathfrak{e}^{\text{\rm full}}_{s(\gamma)}$ into the ideal
$J^{\text{\rm full}}_{r(\gamma)}=\text{\rm ker}\,\mathfrak{e}^{\text{\rm full}}_{r(\gamma)}$.
In particular, \eqref{ad-full}
induces a unique linear map
$\mathfrak{a}^{\text{\rm full}}_{\boldsymbol{\omega}}:C^*(\text{\rm IntIso}(\mathcal{G})s(\gamma))\to
C^*(\text{\rm IntIso}(\mathcal{G})r(\gamma))$, which makes a commutative diagram
$$
\xymatrix{
C^*(\text{\rm IntIso}(\mathcal{G}))\ar[r]^{\text{\rm ad}^{\text{\rm full}}_n}\ar[d]_{\mathfrak{e}^{\text{\rm full}}_{s(\gamma)}}
&
C^*(\text{\rm IntIso}(\mathcal{G}))\ar[d]^{\mathfrak{e}^{\text{\rm full}}_{r(\gamma)}}\\
C^*(\text{\rm IntIso}(\mathcal{G})s(\gamma))
\ar[r]_{\mathfrak{a}^{\text{\rm full}}_{\boldsymbol{\omega}}}
&C^*(\text{\rm IntIso}(\mathcal{G})r(\gamma))
}
$$
\item[(ii)] Likewise (when working with the inclusion $C^*_{\text{\rm red}}(\text{\rm IntIso}(\mathcal{G}))
\subset C^*_{\text{\rm red}}(\mathcal{G})$), the map
\begin{equation}
\text{\rm ad}^{\text{\rm red}}_n:C^*_{\text{\rm red}}(\text{\rm IntIso}(\mathcal{G}))\ni a\longmapsto
nan^*\in  C^*_{\text{\rm red}}(\text{\rm IntIso}(\mathcal{G}))
\label{ad-red}
\end{equation}
sends
$J^{\text{\rm other}}_{s(\gamma)}=\text{\rm ker}\,\mathfrak{e}^{\text{\rm other}}_{s(\gamma)}$ into
$J^{\text{\rm other}}_{r(\gamma)}=\text{\rm ker}\,\mathfrak{e}^{\text{\rm other}}_{r(\gamma)}$, so 
\eqref{ad-red}
induces a unique linear map
$\mathfrak{a}^{\text{\rm other}}_{\boldsymbol{\omega}}:C^*_{\text{\rm other},s(\gamma)}(\text{\rm IntIso}(\mathcal{G})s(\gamma))\to
C^*_{\text{\rm other},r(\gamma)}(\text{\rm IntIso}(\mathcal{G})r(\gamma))$, which makes a commutative diagram
$$
\xymatrix{
C^*_{\text{\rm red}}(\text{\rm IntIso}(\mathcal{G}))\ar[r]^{\text{\rm ad}^{\text{\rm red}}_n}
\ar[d]_{\mathfrak{e}^{\text{\rm other}}_{s(\gamma)}}
&
C^*_{\text{\rm red}}(\text{\rm IntIso}(\mathcal{G}))\ar[d]^{\mathfrak{e}^{\text{\rm other}}_{r(\gamma)}}\\
C^*_{\text{\rm other},s(\gamma)}(\text{\rm IntIso}(\mathcal{G})s(\gamma))
\ar[r]_{\mathfrak{a}^{\text{\rm other}}_{\boldsymbol{\omega}}}
&
C^*_{\text{\rm other},r(\gamma)}(\text{\rm IntIso}(\mathcal{G})r(\gamma))
}
$$
\item[(iii)] The map \eqref{ad-red} sends 
the ideal $\text{\rm ker}p^{\mathbb{E}_{\text{\rm red}}}_{s(\gamma)}
=\text{\rm ker}\,\mathfrak{e}^{\text{\rm red}}_{s(\gamma)}$ into the ideal
$\text{\rm ker}p^{\mathbb{E}_{\text{\rm red}}}_{r(\gamma)}
=\text{\rm ker}\,\mathfrak{e}^{\text{\rm red}}_{r(\gamma)}$, so it also induces a unique linear map
$\mathfrak{a}^{\text{\rm red}}_{\boldsymbol{\omega}}:C^*_{\text{\rm red}}(\text{\rm IntIso}(\mathcal{G})s(\gamma))\to
C^*_{\text{\rm red}}(\text{\rm IntIso}(\mathcal{G})r(\gamma))$, which makes a commutative diagram
$$
\xymatrix{
C^*_{\text{\rm red}}(\text{\rm IntIso}(\mathcal{G}))\ar[r]^{\text{\rm ad}^{\text{\rm red}}_n}
\ar[d]_{\mathfrak{e}^{\text{\rm red}}_{s(\gamma)}}
&
C^*_{\text{\rm red}}(\text{\rm IntIso}(\mathcal{G}))\ar[d]^{\mathfrak{e}^{\text{\rm red}}_{r(\gamma)}}\\
C^*_{\text{\rm red}}(\text{\rm IntIso}(\mathcal{G})s(\gamma))
\ar[r]_{\mathfrak{a}^{\text{\rm red}}_{\boldsymbol{\omega}}}
&
C^*_{\text{\rm red}}(\text{\rm IntIso}(\mathcal{G})r(\gamma))
}
$$
\item[(iv)] All maps $\mathfrak{a}^{\text{\rm full}}_{\boldsymbol{\omega}}$,
$\mathfrak{a}^{\text{\rm other}}_{\boldsymbol{\omega}}$, $\mathfrak{a}^{\text{\rm red}}_{\boldsymbol{\omega}}$
are $*$-isomorphisms, with inverses 
$\mathfrak{a}^{\text{\rm full}}_{\boldsymbol{\omega}^{\text{\rm op}}}$,
$\mathfrak{a}^{\text{\rm other}}_{\boldsymbol{\omega}^{\text{\rm op}}}$
$\mathfrak{a}^{\text{\rm red}}_{\boldsymbol{\omega}^{\text{\rm op}}}$, respectively, are the corresponding maps associated with the triple
$\boldsymbol{\omega}^{\text{\rm op}}=(\gamma^{-1},\mathcal{B}^{-1},n^*)$.
\end{itemize}
\end{lemma}

\begin{proof}
\begin{nagyrevenv}
We prove statements (i) and (ii) simultaneously.
Let $\mathfrak{D}\subset\mathfrak{A}$ denote either one of the inclusions 
$C^*(\text{IntIso}(\mathcal{G}))\subset C^*(\mathcal{G})$, or
$C^*_{\text{red}}(\text{IntIso}(\mathcal{G}))\subset C^*_{\text{red}}(\mathcal{G})$.
The desired staments follows from Lemma \ref{central-norm-ideal} applied to the inclusions
$C_0(\mathcal{G}^{(0)})\subset \mathfrak{D}\subset\mathfrak{A}$, once we show that
%
\begin{equation}
\text{ad}_n\left(C_{0,s(\gamma)}(\mathcal{G}^{(0)})\right)\subset
C_{0,r(\gamma)}(\mathcal{G}^{(0)}).
\label{ad-n-C0}
\end{equation}
However, this statement is immediate from \eqref{nfn*} (see also Remark \ref{rem-nfn*}).

Statement (iii) now clearly follows from \eqref{ad-n-C0}, since by \eqref{Enormalized} we also know that that
$$\mathbb{E}_{\text{red}}(nan^*)=n\mathbb{E}_{\text{red}}(a)n^*,\,\,\,\forall\,
a\in C^*_{\text{red}}(\mathcal{G}).$$

Statement (iv) is also pretty clear, since (using the unified notation as above)
\begin{align*}
\text{ad}_n\circ \text{ad}_{n^*}&=\text{ad}_{nn^*}=\text{multiplication by the central element }(nn^*)^2\text{ on }\mathfrak{B}\text{, and}\\
\text{ad}_{n^*}\circ \text{ad}_{n}&=\text{ad}_{n^*n}=\text{multiplication by the central element }(n^*n)^2\text{ on }\mathfrak{B},
\end{align*}
and for any one of the $*$-homomorphisms $\mathfrak{e}^{\dots}_{s(\gamma)}$ and
$\mathfrak{e}^{\dots}_{r(\gamma)}$, the elements $\mathfrak{e}^{\dots}_{s(\gamma)}(n^*n)$
and $\mathfrak{e}^{\dots}_{r(\gamma)}(nn^*)$ act as the units in the unital $C^*$-algebras
$\mathfrak{e}^{\dots}_{s(\gamma)}(\mathfrak{B})$ and
$\mathfrak{e}^{\dots}_{r(\gamma)}(\mathfrak{B})$.
\end{nagyrevenv}
\end{proof}

\begin{mycomment}
According to \cite[Thm 3.1]{BNRSW}, {\em the inclusion
$C^*_{\text{red}}(\text{\rm IntIso}(\mathcal{G}))\subset C^*_{\text{red}}(\mathcal{G})$ is always essential}, so
by applying  Proposition \ref{reg-simple} and the other results from Section
\ref{min-sec} and we obtain the following simplicity criterion.
\end{mycomment}

\begin{proposition}\label{red-simple-prop}
The reduced $C^*$-algebra $C^*_{\text{\rm red}}(\mathcal{G})$ is simple, if and only if the inclusion
$C^*_{\text{\rm red}}(\text{\rm IntIso}(\mathcal{G}))\subset C^*_{\text{\rm red}}(\mathcal{G})$ is minimal; equivalently, the only ideals $J\unlhd C^*_{\text{\rm red}}(\text{\rm IntIso}(\mathcal{G}))$, satisfying
\begin{equation}
nJn^*\subset J, \,\,\,\forall\,n\in\mathfrak{N}(\mathcal{G}),
\label{red-simple-prop-eq}
\end{equation}
are the trivial ideals $J=\{0\}$ and $J=C^*_{\text{\rm red}}(\text{\rm IntIso}(\mathcal{G}))$.
\end{proposition}

We are now in position to formulate one of the main results in this paper.

\begin{theorem}\label{simplicity-for-red}
Assume there is some unit $u_0\in\mathcal{G}^{(0)}_{\text{\rm r-cont}}$, for which
the (discrete countable) group $\text{\rm IntIso}(\mathcal{G})u_0$ is $C^*$-simple. Then the following are equivalent:
\begin{itemize}
\item[(i)] $\mathcal{G}$ is minimal;
\item[(ii)] $C^*_{\text{\rm red}}(\mathcal{G})$ is simple.
\end{itemize}
\end{theorem}

\begin{proof}
By Proposition \ref{reg-simple}, the implication $(ii)\Rightarrow (i)$ is trivial (and holds true without any restrictions).

$(i)\Rightarrow (ii)$. Assume $\mathcal{G}$ is minimal, and let us prove the simplicity of
$C^*_\text{red}(\mathcal{G})$.

With Proposition \ref{red-simple-prop} in mind, all we need to prove is that the inclusion 
$C^*_{\text{red}}(\text{IntIso}(\mathcal{G}))\subset C^*_{\text{red}}(\mathcal{G})$
is minimal.
Fix some ideal $J\lhd C^*_{\text{red}}(\text{IntIso}(\mathcal{G}))$ satisfying \eqref{red-simple-prop-eq}, and let us prove that $J=\{0\}$. By the minimality assumption on $\mathcal{G}$ (cf. Remark \ref{rem-nfn*}), the set
$\mathcal{X}=r(\mathcal{G}u_0)$ is dense in $\mathcal{G}^{(0)}$. By Lemma \ref{lemma-eu}, it follows that
$\mathcal{X}\subset\mathcal{G}^{(0)}_{\text{r-cont}}$, and furthermore,
all groups $\text{IntIso}(\mathcal{G})u$, $u\in\mathcal{X}$, (whose reduced $C^*$-algebras are $*$-isomorphic
to $C^*_{\text{red}}(\text{IntIso}(\mathcal{G})u_0)$) are $C^*$-simple. 
By Remark \ref{G-isotropic-consequence}.C, it follows that the inclusion
$C_0(\mathcal{G}^{(0)})\subset C^*_{\text{red}}(\text{IntIso}(\mathcal{G}))$ is essential, so in order to prove that
$J=\{0\}$ it suffices to prove that the ideal $J_0=J\cap C_0(\mathcal{G}^{(0)})\unlhd C_0(\mathcal{G}^{(0)})$ is the zero ideal $\{0\}$. However, this follows immediately from the minimality of $\mathcal{G}$, since we clearly have $nJ_0n^*\subset J_0$, $\forall\,n\in\mathfrak{N}(\mathcal{G})$.
\end{proof}

\begin{mycomment}
At this point we are unable to provide any converse statement of Theorem \ref{simplicity-for-red} (i.e. the fact that its hypothesis is also {\em necessary\/} condition for the simplicity of $C^*_\text{red}(\mathcal{G})$).  However, based on Proposition \ref{aug-ideal}, the following
{\em non-simplicity\/} test is available.
\end{mycomment}

\begin{proposition}\label{red-simple-converse-other}
If $C^*_{\text{\rm red}}(\mathcal{\mathcal{G}})$ is simple, and there exists a unit $u_0\in\mathcal{G}^{(0)}$, such that the augmentation
homomorphism $\mathbb{C}[\text{\rm IntIso}(\mathcal{G}u_0)]\to\mathbb{C}$ is continuous relative to the norm
$\|\,.\,\|_{\text{\rm other},u_0}$, then 
$\text{\rm IntIso}(\mathcal{G})u=\{e\}$, $\forall\,u\in\mathcal{G}^{(0)}$.
\end{proposition}
\begin{proof}
First, by the assumption on $u_0$ and by Proposition \ref{aug-ideal}, we know that the augmentation ideal
$\mathfrak{V}_{\text{red}}(\text{IntIso}(\mathcal{G}))$ is a proper ideal in 
$C^*_{\text{red}}(\text{IntIso}(\mathcal{G}))$.
Secondly, since we clearly have
$$n\mathfrak{V}_{\text{red}}(\text{IntIso}(\mathcal{G}))n^*\subset
\mathfrak{V}_{\text{red}}(\text{IntIso}(\mathcal{G})),\,\,\,\forall\,n\in\mathfrak{N}(\mathcal{G}),$$ 
by Proposition \ref{red-simple-prop}, the simplicity of $C^*_{\text{red}}(\mathcal{G})$ forces
$\mathfrak{V}_{\text{red}}(\text{IntIso}(\mathcal{G}))=\{0\}$, which in turn by Proposition \ref{aug-ideal} forces
$\mathcal{G}^{(0)}=\text{IntIso}(\mathcal{G})$.
\end{proof}

\begin{corollary}\label{red-simple-converse-amenable}
If $C^*_{\text{\rm red}}(\mathcal{G})$ is simple, and there is a unit $u_0\in\mathcal{G}^{(0)}$, such that 
$\text{\rm IntIso}(\mathcal{G})u_0$ is amenable, then
$\text{\rm IntIso}(\mathcal{G})u=\{e\}$, $\forall\,u\in\mathcal{G}^{(0)}$.
\end{corollary}

\begin{remark}\label{proper-amenable}
If $u_0\in \mathcal{G}^{(0)}$ satisfying the hypothesis of Proposition \ref{red-simple-converse-other}
is a \nagyrev{\em unit of continuous reduction relative to $\mathbb{E}_\text{\rm red}$}, then $\text{IntIso}(\mathcal{G})u_0$ is necessarily amenable. This is due to the well known fact
the augmentation homomorphism on a group algebra $\mathbb{C}[H]$ is $\|\,.\,\|_{\text{red}}$-continuous, if and only if $H$ is amenable.
\end{remark}

Our methodology also applies to full groupoid $C^*$-algebras,  for which we recover the following result from
\cite{BCFS}. 

\begin{corollary}{(cf. \cite[Thm. 5.1]{BCFS})}\label{full-simple-thm}
The full $C^*$-algebra $C^*(\mathcal{G})$ is simple, if and only if all three conditions below are satisfied
\begin{itemize}
\item[(i)] $\mathcal{G}$ is minimal;
\item[(ii)] $\pi_{\text{\rm red}}:C^*(\mathcal{G})\to C^*_{\text{\rm red}}(\mathcal{G})$ is an isomorphism;
\item[(iii)] $\text{\rm IntIso}(\mathcal{G})u=\{e\}$, $\forall\,u\in\mathcal{G}^{(0)}$.
\end{itemize}
\end{corollary}
\begin{proof}
The implication ``(i) and (ii) and (iii)'' $\Rightarrow$ ``$C^*(\mathcal{G})$ simple'' is obvious.

Conversely, the implication ``$C^*(\mathcal{G})$ simple'' $\Rightarrow$ ``(i) and (ii)'' is also obvious, so we only need to justify the implication ``$C^*(\mathcal{G})$ simple'' $\Rightarrow$ ``(iii).''
To this end, we argue again as above, by observing that, if $C^*(\mathcal{G})$ is simple, then the inclusion 
$C^*(\text{IntIso)}(\mathcal{G})\subset C^*(\mathcal{G})$ is minimal, so
the full augmentation (proper!) ideal $\mathfrak{V}_{\text{full}}(\text{IntIso}(\mathcal{G}))
\lhd C^*(\text{IntIso}(\mathcal{G}))$, which also satisfies
$$n\mathfrak{V}_{\text{full}}(\text{IntIso}(\mathcal{G}))n^*\subset
\mathfrak{V}_{\text{full}}(\text{IntIso}(\mathcal{G})),\,\,\,\forall\,n\in\mathfrak{N}(\mathcal{G}),$$ 
must be the zero ideal $\mathfrak{V}_{\text{full}}(\text{IntIso}(\mathcal{G}))=\{0\}$, which again implies
$\mathcal{G}^{(0)}=\text{IntIso}(\mathcal{G})$.
\end{proof}

\begin{mycomment}
By the density of $\mathcal{G}^{(0)}_\circ$ in $\mathcal{G}^{(0)}$, the condition
\begin{equation}
\text{IntIso}(\mathcal{G})u=\{e\},\,\,\,\forall\,u\in \mathcal{G}^{(0)},
\end{equation}
that appears in Proposition \ref{red-simple-converse-other}, as well as in
Corollary \ref{red-simple-converse-amenable} and Theorem \ref{full-simple-thm},
is equivalent to the condition  that the groupoid $\mathcal{G}$ is {\em topologically principal}, in the sense that
the set 
$$\left\{u\in\mathcal{G}^{(0)}\,:\,\text{Iso}(\mathcal{G})u=\{e\}\right\}$$ 
is {\em dense\/}
in $\mathcal{G}^{(0)}$.

Of course, if some unit $u_0\in \mathcal{G}^{(0)}$ has trivial isotropy $\text{Iso}(\mathcal{G})u_0=\{e\}$, then
it also satisfies the hypothesis of Theorem \ref{simplicity-for-red}, since for any unit $u$,
$\text{IntIso}(\mathcal{G})u$ is a 
normal subgroup of $\text{Iso}(\mathcal{G})u$.
\end{mycomment}

\section{Applications to \'etale transformation groupoids}\label{cross-prod-sec}

The results from the preceding section specialize nicely to {\em transformation groupoids}, which are \'etale groupoids constructed as follows. One starts with a (traditional) dynamical system 
$G\curvearrowright Q$, which consists of a countable discrete group $G$ acting by homeomorphisms on a second countable  locally compact Hausdorff space $Q$. (For any $g\in G$, the corresponding homeomorphism of $Q$ will be simply denoted by $q\longmapsto gq$.) Out of this data, one equips the product space $\mathcal{G}=G\times Q$ with the product topology and the groupoid structure obtained by defining the space of composable pairs to be
$$\mathcal{G}^{(2)}=\left\{\big((g_1,q_1),(g_2,q_2)\big)\,:\,g_1,g_2\in G,\,q_1,q_2\in Q,\,q_1=g_2q_2\right\},$$ and the composition operation defined by $(g_1,q_1)\cdot (g_2,q_2)=(g_1g_2,q_2)$. The inverse operation is $(g,q)\longmapsto (g^{-1},gq)$, while the unit space $\mathcal{G}^{(0)}=\{e\}\times Q$ is of course identified with $Q$. By this identification, the source and range maps $s,r:G\times Q\to Q$ are simply defined as $s(g,q)=q$ and $r(g,q)=gq$.
For any element $\gamma=(g,q)\in\mathcal{G}$, the set $\{g\}\times Q$ is obviously an open bisection containing $\gamma$.

The full and the reduced $C^*$ algebras of the \'etale transformation groupoid $\mathcal{G}$ are identified with the full and the reduced {\em crossed product $C^*$-algebras\/}
$C_0(Q)\rtimes G$ and $C_0(Q)\rtimes_{\text{red}}G$, respectively, which are constructed as follows.
First of all, one considers the direct sum
$\bigoplus_{g\in G}C_0(Q)$, which consists of all $G$-tuples $(a_g)_{g\in G}$ of functions in $C_0(Q)$, with $a_g=0$, for all but finitely many $g\in G$, and we endow it with a $*$-algebra structure defined as follows
\begin{itemize}
\item $(a\times b)_g(q)=\sum_{\substack{g_1,g_2\in G\\ g_1g_2=g}}a_{g_1}(q)b_{g_2}(g_1^{-1}q)$,
$\forall\,(g,q)\in G\times Q,\,a=(a_g)_{g\in G},\,b=(b_g)_{g\in G}\in \bigoplus_{g\in G}C_0(Q)$;
\item $(a^*)_g(q)=\overline{a_{g^{-1}}(g^{-1}q)}$, $\forall\,(g,q)\in G\times Q, \,a=(a_g)_{g\in G}\in \bigoplus_{g\in G}C_0(Q)$.
\end{itemize}
In order to distinguish between these $*$-algebra operations and the usual operations in the direct sum $*$-algebra, we will denote this new $*$-algebra structure by $C_0(Q)[G]$. One defines the full $C^*$-norm on $C_0(Q)[G]$ by 
\begin{equation}
\|a\|_{\text{full}}=\sup\left\{\|\pi(a)\|\,:\,\pi\text{ non-degenerate $*$-representation of $C_0(Q)[G]$}\right\}.
\label{def-full-norm-cp}
\end{equation}
The above definition is slightly different than the standard one found in the literature, which involves the so-called {\em covariant representations\/} of the $C^*$-dynamical system $(C_0(Q),G,\lambda)$, where
$\lambda:G\to Aut(C_0(Q))$ is the action given by $(\lambda_gf)(q)=f(g^{-1}q)$, $g\in G$, $q\in Q$, $f\in C_0(Q)$.
However, it is fairly easy to show that,  for any non-degenerate $*$-representation
$\pi:C_0(Q)[G]\to\mathscr{B}(\mathscr{H})$ there exists a unique unitary representation
$G\ni g\longmapsto U_g^\pi\in\mathscr{U}(\mathscr{H})$, satisfying the identity
\begin{equation}
\pi(a)=\sum_{\substack{g\in G\\ a_g\neq 0}}(\pi\circ \chi_e)(a_g) U_g^\pi,
\,\,\,\forall\,a=(a_g)_{g\in G}\in C_0(Q)[G],\label{cov-rep}
\end{equation}
where $\chi_e$ is the $*$-homomorphism $C_0(Q)\ni f\longmapsto (\delta_{\nagyrev{e,g}}f)_{g\in G}\in C_0(Q)[G]$.

With this set-up in mind, the full crossed product $C^*$-algebra $C_0(Q)\rtimes G$ is the completion of
$C_0(Q)[G]$ in the $C^*$-norm $\|\,.\,\|_{\text{full}}$. 
The $*$-homomorphism $\chi_e$ gives rise to a non-degenerate inclusion 
$\chi_e:C_0(Q)\hookrightarrow C_0(Q)\rtimes G$; with the help of
\eqref{cov-rep} -- applied to the universal representation of $C_0(Q)[G]$, this gives rise to a group representation
$G\ni g\longmapsto \mathbf{u}_g\in\mathcal{U}(M(C_0(Q)\rtimes G))$ (the unitary group of the multiplier algebra), 
which allows us to present the inclusion $C_0(Q)[G]\subset C_0(Q)\rtimes G$ as
$$C_0(Q)[G]\ni a=(a_g)_{g\in G}\longmapsto 
\sum_{\substack{g\in G\\ a_g\neq 0}}\chi_e(a_g)\mathbf{u}_g\in C_0(Q)\rtimes G.$$
For simplicity, we'll agree to ignore $\chi_e$ from our notation (i.e. to replace $\chi_e(f)$ simply by $f$, thus viewing $C_0(Q)$ as a $C^*$-subalgebra in $C_0(Q)\rtimes G$), so we can simply view
\begin{equation}
C_0(Q)\rtimes G=\overline{\text{span}}\{f\mathbf{u}_g\,:\,f\in C_0(Q),\,g\in G\},
\label{cross-prod=span}
\end{equation}
subject to the product and adjoint rules:
\begin{align}
(f_1\mathbf{u}_{g_1})(f_2\mathbf{u}_{g_2})&=\left(f_1(\lambda_{g_1}f_2)\right)\mathbf{u}_{g_1g_2},
\,\,\,f_1,f_2\in C_0(Q),\,g_1,g_2\in G;\\
(f\mathbf{u}_g)^*&=(\lambda_{g^{-1}}f))\mathbf{u}_{g^{-1}},
\,\,\,f\in C_0(Q),\,g\in G;
\end{align}
(In \eqref{cross-prod=span}, the linear span -- without closure --  is $C_0(Q)[G]$.) 

Just as non-degenerate $*$-representations of $C_0(Q)[G]$ are in bijective correspondence to those of
$C_0(Q)\rtimes G$, the same can be said about non-degenerate {\em $*$-homomorphisms\/} 
(in the sense of Definition \ref{def-non-deg}; a $*$-homomorphism $\Psi:C_0(Q)[G]\to A$ is non-degenerate, if $\Psi\big|_{C_0(Q)}:C_0(Q)\to A$
is such.) In particular, using Remark \ref{MBsubMA-rem}, every non-degenerate $*$-homomorphism
$\Phi:C_0(Q)\rtimes G\to A$ yields a group homomorphism 
$G\ni g\longmapsto \mathbf{u}^\Phi_g=M\Phi(\mathbf{u}_g)
\in\mathcal{U}(M(A))$,
so that 
\begin{equation}
\Phi(f\mathbf{u}_g)=\Phi(f)\mathbf{u}^\Phi_g,\,\,\,
\forall\,f\in C_0(Q),\,g\in G.
\label{uPhi}
\end{equation}

Since $(C_0(Q)[G],\|\,.\,\|_{\text{full}})$ contains  as a dense $*$-subalgebra
$$C_c(Q)[G]=\left\{a=(a_g)_{g\in G}\in C_0(Q)[G],\,\,a_g\in C_c(Q),\,\,\forall\,g\in G\right\},$$
the non-degenerate $*$-representations of (or $*$-homomorphisms defined on) $C_0(Q)[G]$ are in a bijective correspondence (by restriction) to the non-degenerate $*$-representations of (or $*$-homomorphisms defined on) $C_c(Q)[G]$; in other words, we can also view the full crossed product $C_0(Q)\rtimes G$
as the completion of $C_c(Q)[G]$ with respect to the $C^*$-norm
\begin{equation}
\|a\|_{\text{full}}=\sup\left\{\|\pi(a)\|\,:\,\pi\text{ non-degenerate $*$-representation of $C_c(Q)[G]$}\right\}.
\label{def-full-norm-cp2}
\end{equation}
The point here is the fact that the $*$-algebra $C_c(Q)[G]$ is $*$-isomorphic to the convolution 
$*$-algebra $C_c(\mathcal{G})$ of our transformation groupoid $\mathcal{G}$. Explicitly, this $*$-isomorphism
$\Upsilon:C_c(\mathcal{G})\to C_c(Q)[G]$ assigns to every function $f\in C_c(G\times Q)$ the $G$-tuple
$a=(a_g)_{g\in G}\in C_c(Q)[G]$ given by $a_g(q)=f(g,g^{-1}q)$. By completion -- using the equalities
 \eqref{def-full-norm-cp} and \eqref{def-full-norm-cp2}, this gives rise to a
$*$-isomorphism $\Upsilon_{\text{full}}:C^*(\mathcal{G})\to C_0(Q)\rtimes G$.

The linear surjection 
\begin{equation}
C_0(Q)[G]\ni (a_g)_{g\in G}\longmapsto a_e\in C_0(Q)
\label{exp-0-cp}
\end{equation}
 is $\|\,.\,\|_{\text{full}}$-contractive, thus
it gives rise to conditional expectation $\mathbb{E}^{\rtimes}$ of $C_0(Q)\rtimes G$ onto $C_0(Q)$, which acts on the generators as
$$\mathbb{E}^\rtimes (f\mathbf{u}_g)=\delta_{\nagysecondrev{e,g}}f,\,\,\,f\in C_0(Q),\,g\in G.$$
 Using the identification 
between $Q$ and the unit space $\mathcal{G}^{(0)}$ of our transformation groupoid $\mathcal{G}$, we have a commutative diagram
$$
\xymatrix{
C^*(\mathcal{G}) \ar[rr]^{\Upsilon_{\text{full}}\,\,\,\,}_{\sim\,\,\,\,\,}\ar[d]_{\mathbb{E}^{\mathcal{G}}}
& & C_0(Q)\rtimes G\ar[d]^{\mathbb{E}^{\rtimes}}\\
C_0(\mathcal{G}^{(0)})
\ar@{=}[rr] 
& & C_0(Q)
}
$$
When we apply the KSGNS construction, the above diagram will yield a $*$-isomorphism
$\Upsilon_{\text{red}}:C^*_{\text{red}}(\mathcal{G})\to
C_0(Q)\rtimes_{\text{red}} G$ between the reduced $C^*$-algebra of $\mathcal{G}$ and the reduced crossed product $C^*$-algebra. Equivalently, one can view $C_0(Q)\rtimes_{\text{red}}G$ as the completion of
$C_c(Q)[G]$ with respect with the unique $C^*$-norm $\|\,.\,\|_{\text{red}}(\leq\|\,.\,\|_{\text{full}})$ that makes the $*$-isomorphism $\Upsilon:(C_c(\mathcal{G}),\|\,.\,\|_{\text{red}})\to 
(C_c(Q)[G],\|\,.\,\|_{\text{red}})$ isometric; one then has two commutative diagrams
$$
\xymatrix{
C^*(\mathcal{G}) \ar[rr]^{\Upsilon_{\text{full}}\,\,\,\,}_{\sim\,\,\,\,\,}
\ar[d]_{\pi_{\text{red}}^{\mathcal{G}}}
& & C_0(Q)\rtimes G\ar[d]^{\pi_{\text{red}}^{\rtimes}}\\
C^*_{\text{red}}(\mathcal{G}) \ar[rr]^{\Upsilon_{\text{red}}\,\,\,\,}_{\sim\,\,\,\,\,}
\ar[d]_{\mathbb{E}_{\text{red}}^{\mathcal{G}}}
& & C_0(Q)\rtimes _{\text{red}} G\ar[d]^{\mathbb{E}_{\text{red}}^{\rtimes}}\\
C_0(\mathcal{G}^{(0)})
\ar@{=}[rr] 
& & C_0(Q)
}
$$
where $\pi^\rtimes_{\text{red}}:C_0(Q)\rtimes G\to C_0(Q)\rtimes_{\text{red}} G$ is the quotient $*$-homomorphism, and $\mathbb{E}^\rtimes_{\text{red}}:C_0(Q)\rtimes_{\text{red}}\to C_0(Q)$ is the (unique) conditional expectation
arising from \eqref{exp-0-cp} (which is also $\|\,.\,\|_{\text{red}}$-contractive), that yields the factorization
$\mathbb{E}^\rtimes=\mathbb{E}^\rtimes_{\text{red}}\circ\pi^\rtimes_{\text{red}}$.
(Again, with the help of the $*$-homomorphism $\chi_e^{\text{red}}=\pi^\rtimes_{\text{red}}\circ\chi_e:C_0(Q)\hookrightarrow C_0(Q)\rtimes_{\text{red}}G$, we view $C_0(Q)$ as a non-degenerate $C^*$-subalgebra of
$C_0(Q)\rtimes_{\text{red}}G$, and omit $\chi^{\text{red}}_e$ from our notation.)

Lastly, by applying the construction of the unitaries satisfying \eqref{uPhi} to the surjective (thus non-degenerate) $*$-homomorphism
$\pi^\rtimes_{\text{red}}:C_0(Q)\rtimes G\to C_0(Q)\rtimes_{\text{red}}G$, we also obtain a group homomorphism
$G\ni g\longmapsto \mathbf{u}^{\text{red}}_g\in \mathcal{U}(M(C_0(Q)\rtimes_{\text{red}}G))$
satisfying
\begin{equation}
\pi^\rtimes_{\text{red}}(f\mathbf{u}_g)= f\mathbf{u}^{\text{red}}_g,\,\,\,
\forall\,f\in C_0(Q),\,g\in G.
\label{ured}
\end{equation}

\begin{remark}\label{GV-rem}
A point $(g,q)\in G\times Q=\mathcal{G}$ belongs to the isotropy groupoid $\text{Iso}(\mathcal{G})$, if and only if
$gq=q$; in other words, $g$ belongs to the {\em stabilizer group\/} $G_q=\{g\in G\,:\,gq=q\}$.
Furthermore, using the fact that $G$ comes equipped with the discrete topology, $(g,q)$ belongs to the interior $\text{IntIso}(\mathcal{G})$, if and only if $g\in G_{q'}$, for all $q'$ in some neighborhood of $q$. In other words,
if we define, for any open set $V\subset Q$, the {\em set stabilizer group\/} 
$$G_V=\{g\in G\,:\,gq=q,\,\,\forall\,q\in V\},$$
then, when we identify $Q$ with the unit space $\mathcal{G}^{(0)}$ of our transformation groupoid $\mathcal{G}$, for any $q\in Q$, the group $\text{IntIso}(\mathcal{G})q$ is equal to the {\em interior stabilizer group}
$$G^\circ_q=\bigcup_{\substack{V\text{ open}\\ V\ni q}}G_V.$$
\end{remark}

\begin{mycomment}
When translating the groupoid dictionary to dynamical systems, the following slightly different terminology is employed.
\begin{itemize}
\item[(a)]
The transformation groupoid
$\mathcal{G}=G\times Q $ is minimal (as in Remark \ref{rem-nfn*}), if and only 
$G\curvearrowright Q$ is a {\em minimal action}, in the sense that
the only open/closed
$G$-invariant subsets $S\subset Q$ are $S=\varnothing,\,Q$.
\item[(b)]
The transformation groupoid
$\mathcal{G}=G\times Q $ is topologically principal (as in condition (iii) from Corollary \ref{full-simple-thm}), if and only 
the action is $G\curvearrowright Q$ is a {\em topologically free}, in the sense that
 the set
$\big\{q\in Q\,:\,G_q=\{e\}\big\}$ is dense in $Q$
\end{itemize}
As pointed out in the Comment that followed Remark \ref{rem-nfn*}, minimality for the action
$G\curvearrowright Q$ is always a necessary condition for the simplicity of either
$C_0(Q)\rtimes G$ or $C_0(Q)\rtimes_\text{red}G$.
\end{mycomment}

When we specialize Corollary \ref{full-simple-thm} to the transformation groupoid, and use the above dictionary,
one recovers the following well-known result of Kawamura and Tomyiama.
\begin{corollary}{(cf. \cite{AS} and \cite[Thm. 4.4.]{KT})}\label{full-simple-cross-prod-thm}
The full crossed product $C^*$-algebra $C_0(Q)\rtimes G$ is simple, if and only if all three conditions below are satisfied
\begin{itemize}
\item[(i)] the action $\mathcal{G}\curvearrowright Q$ is minimal;
\item[(ii)] $\pi_{\text{\rm red}}^\rtimes:C_0(Q)\rtimes G\to C_0(Q)\rtimes_{\text{\rm red}}G$ is an isomorphism;
\item[(iii)] the action $\mathcal{G}\curvearrowright Q$ is topologically free.
\end{itemize}
\end{corollary}

Concerning the simplicity of the reduced crossed product, 
in preparation for our adaptation of the main results from the previous section (Theorem \ref{simplicity-for-red}
and Proposition \ref{red-simple-converse-other}), we begin by clarifying the status of 
the open subgroupoid $\text{IntIso}(\mathcal{G})\subset\mathcal{G}$.

\begin{notation}
For each $g\in G$,  denote the fixed point set $\{q\in Q\,:\,gq=q\}$ by $Q^g$.
\end{notation}

\begin{remark}\label{intiso-cross}
Using the $*$-isomorphism
$\Upsilon:C_c(\mathcal{G})\to C_c(Q)[G]$, and viewing
$C_c(Q)[G]$ as either a $*$-subalgebra in $C_0(Q)\rtimes G$, or in $C_0(Q)\rtimes_{\text{red}}G$,
 the $*$-subalgebra $C_c(\text{IntIso}(\mathcal{G}))\subset C_c(\mathcal{G})$ gets identified with either one of 
the $*$-subalgebras (using the convention $C_c(\varnothing)=\{0\}$):
\begin{align*}
\mathscr{A}&=\sum_{g\in G}C_c(\text{Int}(Q^g))\mathbf{u}_g\subset C_0(Q)\rtimes G,\\
\mathscr{A}_{\text{red}}=\pi^\rtimes_{\text{red}}(\mathscr{A})&=\sum_{g\in G}C_c(\text{Int}(Q^g))\mathbf{u}_g^{\text{red}}
\subset C_0(Q)\rtimes_{\text{red}} G.
\end{align*}
The sums defining these $*$-subalgebras are {\em direct\/} sums. The fact that $\mathscr{A}$ is a $*$-subalgebra
can also be seen (without any reference to the $*$-homomorphism $\Upsilon$) from the following easy observations:
\begin{itemize}
\item[(i)] if $f\in C_c(\text{Int}(Q^g))$, then $f\mathbf{u}_g=\mathbf{u}_gf$;
\item[(ii)] if $f_j\in C_c(\text{Int}(Q^{g_j}))$, $j=1,2$, then $(f_1\mathbf{u}_{g_1})(f_2\mathbf{u}_{g_2})=
\mathbf{u}_{g_1}(f_1f_2)\mathbf{u}_{g_2}
=(f_1f_2)\mathbf{u}_{g_1g_2}$, where $f_1f_2\in C_c(\text{Int}(Q^{g_1})\cap \text{Int}(Q^{g_2}))\subset
C_c(\text{Int}(Q^{g_1g_2}))$.
\end{itemize}
Using either one of the $*$-isomorphisms 
$C^*(\mathcal{G})\xrightarrow{\Upsilon_{\text{full}}} C_0(Q)\rtimes G$, or
$C^*_{\text{red}}(\mathcal{G})\xrightarrow{\Upsilon_{\text{red}}} C_0(Q)\rtimes_{\text{red}} G$, 
 we can identify 
$C^*(\text{IntIso}(\mathcal{G}))$ with the closure $\overline{\mathscr{A}}\subset C_0(Q)\rtimes G$, and
$C_{\text{red}}^*(\text{IntIso}(\mathcal{G}))$ with the closure $\overline{\mathscr{A}_{\text{red}}}\subset
C_0(Q)\rtimes_{\text{red}} G$.
\end{remark}
\begin{remark}
Fix for the moment $q\in Q$ (viewed as a unit in
$\text{IntIso}(\mathcal{G})^{(0)}$).
Following the treatment from Section \ref{bundles-sec}, 
we have a $C^*$-seminorm $p^{\text{full}}_q$ on
$C^*(\text{IntIso}(\mathcal{G}))$ and two $C^*$-seminorms
$p^{\text{other}}_q$, $p^{\mathbb{E}_{\text{red}}}_q$ on $C^*_{\text{red}}(\text{IntIso}(\mathcal{G}))$.
Using the central inclusions 
$$C^*(\text{IntIso}(\mathcal{G}))\supset C_0(Q)\subset C^*_{\text{red}}(\text{IntIso}(\mathcal{G})),$$
 the seminorms $p^{\text{full}}_q$ and $p^{\text{other}}_q$ are defined 
as 
$$b\longmapsto \inf \{ ||fb||: f \in C_{0,q}(Q), 0 \leq f \leq 1=f(q) \},$$
where $b$ belongs to either $C^*(\text{IntIso}(\mathcal{G}))$, or to
$C^*_\text{red}(\text{IntIso}(\mathcal{G}))$.
The seminorm $p^{\mathbb{E}_{\text{red}}}_q$ is defined using the GNS representation
$\Gamma_{\text{ev}_q\circ\mathbb{E}_{\text{red}}}$ of $C^*_{\text{red}}(\text{IntIso}(\mathcal{G}))$ associated with the state  $\text{ev}_q\circ\mathbb{E}_{\text{red}}$. 

When transferring this set-up to $\overline{\mathscr{A}}\simeq C^*(\text{IntIso}(\mathcal{G}))$ and 
$\overline{\mathscr{A}_\text{red}}\simeq C^*_\text{red}(\text{IntIso}(\mathcal{G}))$, we are now dealing with seminorms which we denote
$p^\text{full}_{q\rtimes}$ (on $\overline{\mathscr{A}}$),
$p^\text{other}_{q\rtimes}$ and
$p^\text{red}_{q\rtimes}$ (on $\overline{\mathscr{A}_\text{red}}$).
As in the preceding paragraph, the seminorms $p^\text{full}_{q\rtimes}$ and
$p^\text{other}_{q\rtimes}$ are given by
\begin{align}
p^\text{full}_{q\rtimes}(a)&= \inf \{ ||fa||: f \in C_{0,q}(Q), 0 \leq f \leq 1= f(q) \},\,\,\,a\in
\overline{\mathscr{A}},\label{pfulA}
\\
p^\text{other}_{q\rtimes}(a)&= \inf \{ ||fa||_\text{red}: f \in C_{0,q}(Q), 0 \leq f \leq 1=f(q) \},\,\,\,a\in
\overline{\mathscr{A}_\text{red}},\label{potherA}
\end{align}
while the seminorm $p^\text{red}_{q\rtimes}$ is given by the
GNS representation $\Gamma_{\text{ev}_q\circ\mathbb{E}^\rtimes_{\text{red}}}$ of 
$\overline{\mathscr{A}_\text{red}}$, associated with the state  $\text{ev}_q\circ\mathbb{E}^\rtimes_{\text{red}}\big|_{\overline{\mathscr{A}_\text{red}}}$. 

Alternatively, these seminorms can be realized as follows. Start off with the group algebra
$\mathbb{C}[G^\circ_q]=\mathbb{C}[\text{IntIso}(\mathcal{G})q]$,  denote its canonical unitary generators by
$\{\mathbf{x}_g\}_{g\in G^\circ_q}$, and consider the $*$-homomorphisms (defined for sums indexed by finite sets $F\subset G$)
\begin{align*}
\boldsymbol{\epsilon}_q:\mathscr{A}
\ni\sum_{g\in F}f_g\mathbf{u}_g &\longmapsto \sum_{g\in F}f_g(q)\mathbf{x}_g\in  \mathbb{C}[G^\circ _q],
\\
\boldsymbol{\epsilon}^\prime_q:\mathscr{A}_\text{red}
\ni\sum_{g\in F}f_g\mathbf{u}^\text{red}_g &\longmapsto \sum_{g\in F}f_g(q)\mathbf{x}_g\in  \mathbb{C}[G^\circ _q].
\end{align*}
On the group algebra $\mathbb{C}[G^\circ _q]$ we now have two $C^*$-norms $\|\,.\,\|_\text{full}$ (the full $C^*$-norm) and $\|\,.\,\|_{\text{other},q}\leq\|\,.\,\|_\text{full}$, which by completion
allow us to extend the above $*$-homomorphisms to two 
surjective $*$-homomorphisms, 
$\boldsymbol{\epsilon}^\text{full}_q:\overline{\mathscr{A}}\to C^*(G^\circ _q)$ and
$\boldsymbol{\epsilon}^\text{other}_q:\overline{\mathscr{A}_\text{red}}\to C^*_{\text{other},q}(G^\circ _q)$, so that
\begin{align}
p^\text{full}_{q\rtimes}(a)&=\|\boldsymbol{\epsilon}^\text{full}_q(a)\|_\text{full},\,\,\,\forall\,a
\in\overline{\mathscr{A}};
\\
p^\text{other}_{q\rtimes}(a)&=\|\boldsymbol{\epsilon}^\text{other}_q(a)\|_{\text{other},q},\,\,\,\forall\,a
\in\overline{\mathscr{A}_\text{red}}.
\end{align}
Also, if we equip $\mathbb{C}[G^\circ _q]$ with the reduced $C^*$-norm $\|\,.\,\|_\text{red}\leq
\|\,.\,\|_{\text{other},q}$, by completion, $\boldsymbol{\epsilon}^\prime_q$ also gives rise to a surjective
$*$-homomorphism $\boldsymbol{\epsilon}^\text{red}_q:\overline{\mathscr{A}_\text{red}}\to C^*_\text{red}(G^\circ_q)$,
which allows us to realize
\begin{equation}
p^\text{red}_{q\rtimes}(a)=\|\boldsymbol{\epsilon}^\text{red}_q(a)\|_\text{red},\,\,\,\forall\,a
\in\overline{\mathscr{A}_\text{red}}.
\end{equation}
Using \eqref{norm-sup-punif}, we have
\begin{align}
\|a\|&=\sup_{q\in Q} p^\text{full}_{q\rtimes}(a)=
\sup_{q\in Q}\|\boldsymbol{\epsilon}^\text{full}_q(a)\|_\text{full},\,\,\,a\in\overline{\mathscr{A}};
\label{norm-on-A=sup}
\\
\|a\|_\text{red}&=\sup_{q\in Q} p^\text{other}_{q\rtimes}(a)=
\sup_{q\in Q}\|\boldsymbol{\epsilon}^\text{other}_q(a)\|_{\text{other},q}\,\,\,a\in\overline{\mathscr{A}_\text{red}}.
\end{align}
Using the faithfulness of $\mathbb{E}^\rtimes_\text{red}$, we also have:
\begin{equation}
\|a\|_\text{red}=\sup_{q\in Q} p^\text{red}_{q\rtimes}(a)=
\sup_{q\in Q}\|\boldsymbol{\epsilon}^\text{red}_q(a)\|_{\text{red}}\,\,\,a\in\overline{\mathscr{A}_\text{red}}.
\label{norm-on-Ared=sup}\end{equation}
\end{remark}

\begin{theorem}\label{thm-Delta-incl}
Let $\{\mathbf{v}_g\}_{g\in G}\subset C^*(G)$ and $\{\mathbf{v}^{\text{\rm red}}_g\}_{g\in G}\subset C^*_{\text{\rm red}}(G)$ denote the standard unitary generators if the full, or reduced group $C^*$-algebras, respectively.
The linear maps (defined on sums indexed by finite subsets $F\subset G$)
\begin{align}
\Delta:\mathscr{A}\ni\sum_{g\in F}f_{\nagyrev{g}}\mathbf{u}_g&\longmapsto \sum_{g\in F}f_{\nagyrev{g}}\otimes\mathbf{v}_g\in 
 C_0(Q)\otimes C^*(G),\label{Delta-def}
\\
\Delta^\prime:\mathscr{A}_\text{\rm red}\ni\sum_{g\in F}f_{\nagyrev{g}}\mathbf{u}^{\text{\rm red}}_g&\longmapsto 
\sum_{g\in F}f_{\nagyrev{g}}\otimes\mathbf{v}^{\text{\rm red}}_g\in  C_0(Q)\otimes C_{\text{\rm red}}^*(G)
\label{Delta-red-def}\end{align}
extend to injective $*$-homomorphisms 
$\Delta_\text{\rm full}:\overline{\mathscr{A}}\to C_0(Q)\otimes C^*(G)$ and
$\Delta_{\text{\rm red}}:\overline{\mathscr{A}_{\text{\rm red}}}\to C_0(Q)\otimes C^*_{\text{\rm red}}(G)$.

{\rm (In \eqref{Delta-def} and \eqref{Delta-red-def}, the symbol $\otimes$ denote\nagyrev{s} the maximal tensor product. Of course, since $C_0(Q)$ is nuclear, the maximal and minimal tensor products involved here {\em coincide}.
\nagyrev{We remind the reader that the sums appearing on the left sides involve only functions $f_g\in C_c(Q^g)$.})}  
\end{theorem}

\begin{proof}
First of all, $\Delta$ and $\Delta'$ are clearly $*$-homomorphisms.

Second, consider, for each $q\in Q$ the injective $*$-homomorphisms
$\iota^\text{full}_q:C^*(G^\circ_q)\to C^*(G)$ and 
$\iota^\text{red}_q:C^*_\text{red}(G^\circ_q)\to C^*_\text{red}(G)$ arising from the inclusion
$\mathbb{C}[G^\circ_q]\subset \mathbb{C}[G]$ (which in turn comes from the inclusion $G^\circ_q\subset G$).
Also, denote by $\text{ev}^\text{full}_q:C_0(Q)\otimes C^*(G)\to C^*(G)$ and
$\text{ev}^\text{red}_q:C_0(Q)\otimes C^*_\text{red}(G)\to C^*_\text{red}(G)$ the evaluation maps.
The desired conclusion now follows from Proposition \ref{embed-crit} applied
\begin{itemize}
\item[(a)] to the $*$-homomorphism $\Delta$ and the families $\Sigma^\text{full}=\{C_0(Q)\otimes C^*(G)\xrightarrow{\text{ev}^\text{full}_q}C^*(G)\}_{q\in Q}$, 
$\Psi^\text{full}=\{\overline{\mathscr{A}}\xrightarrow{\iota^\text{full}_q\circ\boldsymbol{\epsilon}^\text{full}_q}C^*(G)\}_{q\in Q}$, and likewise,
\item[(b)] to the the $*$-homomorphism $\Delta'$ and the families $\Sigma^\text{red}=
\{C_0(Q)\otimes C^*_\text{red}(G)\xrightarrow{\text{ev}^\text{red}_q}C^*_\text{red}(G)\}_{q\in Q}$, 
$\Psi^\text{red}=\{\overline{\mathscr{A}_\text{red}}\xrightarrow{\iota^\text{red}_q\circ\boldsymbol{\epsilon}^\text{red}_q}C^*_\text{red}(G)\}_{q\in Q}$.
\end{itemize}
Indeed, on the one hand, 
it is pretty evident that, for a fixed $q\in Q$, the $*$-homomorphisms $\text{ev}^\text{full}_q$ and $\iota^\text{full}_q\circ\boldsymbol{\epsilon}^\text{full}_q$ act the same way on monomials of the form $f\mathbf{u}_g\in\mathscr{A}$,
and likewise,
$\text{ev}^\text{red}_q$ and $\iota^\text{red}_q\circ\boldsymbol{\epsilon}^\text{red}_q$ act same way on
on $f\mathbf{u}^\text{red}_g\in\mathscr{A}_\text{red}$.
(The monomial referred to in each instance involves some $g\in G$ and  $f\in C_c(\text{Int}(Q^g))$.)
One the other hand,
by \eqref{norm-on-A=sup} and \eqref{norm-on-Ared=sup}, both $\Psi^\text{full}$ and $\Psi^\text{red}$ are jointly faithful, while the joint faithfulness of $\Sigma^\text{full}$ and $\Sigma^\text{red}$ is obvious.
\end{proof}

\begin{theorem}\label{t-g-cont-thm}
The isotropic groupoid $\text{\rm IntIso}(\mathcal{G})$ associated to the \'etale transformation group\-oid
$\mathcal{G}=G\times Q$ exhibits the following continuity properties:
\begin{itemize}
\item[(i)] For each $a\in C^*(\text{\rm IntIso}(\mathcal{G}))$, the map
$Q\ni q\longmapsto p^\text{\rm full}_q(a)\in [0,\infty)$ is continuous.
\item[(ii)] For each $a\in C^*_\text{\rm red}(\text{\rm IntIso}(\mathcal{G}))$,
the map
$Q\ni q\longmapsto p^{\mathbb{E}_\text{\rm red}}_q(a)\in [0,\infty)$ is continuous.
In particular, every $q\in Q$ is a \nagyrev{unit of continuous reduction relative to $\mathbb{E}_\text{red}$}, i.e. the $C^*$-seminorms $p^\text{\rm other}_q$ and
$p^{\mathbb{E}_\text{\rm red}}_q$ coincide on $C^*_\text{\rm red}(\text{\rm IntIso}(\mathcal{G}))$, thus the
$C^*$-norms $\|\,.\,\|_{\text{\rm other},q}$ and $\|\,.\,\|_{\text{\rm red}}$ coincide on
$\mathbb{C}[\text{\rm IntIso}(\mathcal{G})q]=\mathbb{C}[G^\circ _q]$.
\end{itemize}
\end{theorem}
\setcounter{myfactnr}{0}

\begin{proof}
The continuity statements from parts (i) and (ii) follow from the following elementary  facts.
\begin{myfact}
For any $C^*$-algebra $B$, the evaluation maps
$\text{\rm ev}_q:C_0(Q)\otimes B\to B$, $q\in Q$ give rise, for each $a\in C_0(Q)\otimes B$,  to a continuous map
$Q\ni q \longmapsto \|\text{\rm ev}_q(a)\|\in [0,\infty)$, which satisfies
$\sup_{q\in Q}\|\text{\rm ev}_q(a)\|=\|a\|$.
\end{myfact}
\begin{myfact}
If $C_0(Q)\subset A$ is a central non-degenerate inclusion, and
$\Delta:A\to C_0(Q)\otimes B$ is an injective $*$-homomorphism, which is $C_0(Q)$-linear
(i.e. $\Delta(fa)=(f\otimes 1)\Delta(a)$, $\forall\,f\in C_0(Q),\,a\in A$), then
$$p^\text{\rm unif}_q(a)=\|(\text{\rm ev}_q\circ\Delta)(a)\|,\,\,\,
\forall\,a\in A,\,q\in Q.$$
In particular, the map $Q\ni q\longmapsto p^\text{\rm unif}_q(a)\in [0,\infty)$ is continuous, for each $a\in A$.
\end{myfact}

As for the second statement from (ii), the equality $p^\text{other}_q=p
^{\mathbb{E}_\text{red}}_q$ follows from Fact 2, applied to the inclusion 
$C^*_\text{red}(\text{IntIso}(\mathcal{G}))\hookrightarrow C_0(Q)\otimes C^*_\text{red}(G)$
obtained from Theorem \ref{thm-Delta-incl}
(by composing $\Delta_\text{red}$ with the isomorphism $C^*_\text{red}(\text{IntIso}(\mathcal{G}))\xrightarrow{\Upsilon_\text{red}}\overline{\mathscr{A}_\text{red}}$), which implies that the group algebra 
inclusion $\mathbb{C}[G^\circ _q]\subset \mathbb{C}[G]$ extends to a $C^*$-algebra inclusion
$\boldsymbol{j}_q:C^*_{\text{other},q}(G^\circ _q)\hookrightarrow C^*_\text{red}(G)$.
(Since we always have an inclusion $C^*_\text{red}(G^\circ _q)\hookrightarrow C^*_\text{red}(G)$, and
$\boldsymbol{j}_q$ factors through the quotient $*$-homomorphism $\boldsymbol{\kappa}_q:C^*_{\text{other},q}(G^\circ _q)
\to C^*_\text{red}(G^\circ _q)$, the injectivity of  $\boldsymbol{j}_q$ indeed forces $\boldsymbol{\kappa}_q$ to be isometric.)
\end{proof}

Since all units are in transformation groupoids are \nagyrev{units of continuous reduction relative to 
$\mathbb{E}_\text{red}$}, when we adapt Theorem \ref{simplicity-for-red} to this setting, we obtain the following.

\begin{corollary}\label{simplicity-for-cpred}
If there is some $q_0\in Q$, such that the group $G^\circ _{q_0}$ is a $C^*$-simple, then the following conditions are equivalent
\begin{itemize}
\item[(i)] the action $G\curvearrowright Q$ is minimal;
\item[(ii)] the reduced crossed product  $C_0(Q)\rtimes_\text{\rm red}G$ is a simple $C^*$-algebra.
\end{itemize}
\end{corollary}

\begin{mycomment}
The above result offers a slight improvement to a Theorem of Ozawa (\cite[Thm. 14(1)]{Ozawa}), in which our hypothesis 
-- $C^*$-simplicity of some interior stabilizer $G^\circ _{q_0}$ -- is replaced by a stronger condition: the $C^*$-simplicity of some full stabilizer group $G_{q_0}$. (After all, for any $q\in Q$, we know that $G^\circ _q$ is a normal subgroup of $G_q$; also it is well known that normal subgroups of $C^*$-simple groups are also $C^*$-simple.)
\end{mycomment}

Once again, since all units are continuously reduced, when adapted to crossed products,
the hypothesis from Proposition \ref{red-simple-converse-other} is indistinguishable to the one 
from Corollary \ref{red-simple-converse-amenable}, so our non-simplicity criterion for crossed products is stated as follows (compare to \cite[Thm. 14(2)]{Ozawa})

\begin{corollary}\label{ozawa-sharpening}
If the reduced crossed product  $C_0(Q)\rtimes_\text{\rm red}G$ is a simple $C^*$-algebra, and there is some $q_0\in Q$, such that the group $G^\circ _{q_0}$ is amenable, then 
$G^\circ _q=\{e\}$, $\forall\,q\in Q$, i.e. the action $G\curvearrowright Q$ is topologically free.
\end{corollary}

\section*{Appendix: On the inclusion \texorpdfstring{$C^*(\text{IntIso}(\mathcal{G}))\subset C^*(\mathcal{G})$}{C*(IntIso(G)) subset C*(G)}}
\setcounter{oldsec}{\value{section}}
\setcounter{section}{1}
\renewcommand\thesection{\Alph{section}}
\setcounter{theorem}{0}
\setcounter{appresult}{0}

As in Section \ref{grp-sec}, we fix an \'etale Hausdorff, second countable groupoid $\mathcal{G}$, so that its full 
$C^*$-algebra is the completion of $(C_c(\mathcal{G}),\times,{}^*)$ in the full norm 
$\|\,.\,\|_{\text{full}}$ defined in \eqref{full-norm}.  

In our analysis of the isotropy groupoid, we will use the well known identification between the following three spaces, associated with a discrete group $H$:
\begin{itemize}
\item[(i)] the space of positive linear functionals on the full group $C^*$-algebra $C^*(H)$;
\item[(ii)] the space of positive linear functionals on the group algebra $\phi:\mathbb{C}[H]\to\mathbb{C}$, i.e. those satisfying $\phi(f^*\times f)\geq 0$, $\forall\,f\in \mathbb{C}[H]$;
\item[(iii)] the space of  {\em positive definite  functions on
$H$}, i.e. those functions $\theta:H\to\mathbb{C}$, with the property that:
{\em for any finite set $\{h_1,\dots,h_n\}\subset H$, the matrix $[\theta(h_i^{-1}h_j)]_{i,j=1}^n\in M_n$ is 
positive}.
\end{itemize}

This correspondence assigns to every positive definite $\theta$ the positive linear functional
$\mathbb{C}[H]\ni f\longmapsto\sum_{h\in H}\theta(h)f(h)\in\mathbb{C}$, which turns out to be 
$\|\,.\,\|_{\text{full}}$-bounded, thus it extends to a unique positive linear functional on  $C^*(H)$.

Another well know elementary fact, is that, whenever $H_0\subset H$ is a subgroup, and $\theta_0$
 is a positive definite function on $H_0$, the function $\theta:H\to\mathbb{C}$ given by
$$\theta(h)=\left\{\begin{array}{l}\theta_0(h)\text{, if }h\in H_0\\ 0\text{, otherwise}\end{array}\right.$$
is again a positive definite function on $H$. Using this, one obtains the well known fact that the canonical inclusion $\mathbb{C}[H_0]\subset \mathbb{C}[H]$ gives rise to two $C^*$-inclusions of group $C^*$-algebras
$C^*(H_0)\subset C^*(H)$ and 
$C^*_{\text{red}}(H_0)\subset C^*_{\text{red}}(H)$. (These can also be justified quickly using
 Proposition \ref{embed-crit}.)

The next two results are well known, but due to their elementary nature, we supply them with proofs.

\begin{applemma}(a variation of \cite[Prop.4.2]{Renault-JOT87})
\label{app-positivity}
If $\phi:C_c(\mathcal{G})\to\mathbb{C}$ is a linear functional, which is positive, in the sense that
\begin{equation}
\phi(f^*\times f)\geq 0,\,\,\,\forall\,f\in C_c(\mathcal{G}),
\label{app-phi-pos}
\end{equation}
then $\phi$ also satisfies the inequalities
\begin{equation}
\phi(f^*\times h^*\times h\times f)\leq \|h\|_{\text{\rm full}}^2\cdot
\phi(f^*\times f),\,\,\,\forall\,f,h\in C_c(\mathcal{G}).
\label{app-phi-bdd}
\end{equation}
In particular, if $\phi\big|_{C_c(\mathcal{G}^{(0)})}$ is $\|\,.\,\|_\infty$-bounded, then
$\phi$ is $\|\,.\,\|_{\text{\rm full}}$-continuous, thus it extends to a positive linear functional
on $C^*(\mathcal{G})$.
\end{applemma}
\begin{proof}
The key step is the following slightly weaker version of 
\eqref{app-phi-bdd}:

\begin{claim}
For every $h\in C_c(\mathcal{G})$, there exists some constant $C_h\geq 0$, such that
\begin{equation}
\phi(f^*\times h^*\times h\times f)\leq C_h\cdot
\phi(f^*\times f),\,\,\,\forall\,f,h\in C_c(\mathcal{G}).
\label{app-phi-bdd-easy}
\end{equation}
\end{claim}
\noindent
Using the fact that $C_c(\mathcal{G})$ is spanned by functions supported on open bisections, and 
 Cauchy-Bunyakovsky-Schwarz -- which implies 
$$\phi(f^*\times(\sum_{j=1}^m h_j)^*\times(\sum_{j=1}^m h_j)\times f)
\leq m^2\sum_{j=1}^m\phi(f^*\times h_j^*\times h_j\times f),$$ 
it suffices to prove the Claim under the additional assumption that $h\in C_c(\mathcal{B})$, for some open bisection $\mathcal{B}\subset \mathcal{G}$.
But in this case, $h^*\times h$ belongs to $C_c(\mathcal{G}^{(0)})$, so by replacing $h$ with 
$(h^*\times h)^{1/2}$ we can in fact assume that $h\in C_c(\mathcal{G}^{(0)})$. But now if we take a
function $k\in C_c(\mathcal{G}^{(0)})$, such that $0\leq k\leq 1$ and $k\big|_{r(\text{supp}\,f)\cup \text{supp}\,h}=1$, then $k\times f=f$, so the function $\|h\|^2_\infty k^*\times k-h^*\times h\in C_c(\mathcal{G}^{(0)})$ is positive, so if we take $g$ to be its square root, we have the identity
$h^*\times h+g^*\times g=\|h\|_\infty^2 k^*\times k$, which yields:
$$f^*\times h^*\times h\times f+
f^*\times g^*\times g\times f
=\|h\|_\infty^2 f^*\times f,$$
which by positivity implies \eqref{app-phi-bdd-easy} with $C_h=\|h\|_\infty ^2$.

Having justified the Claim, we can conclude that, by considering the separate completion $\mathscr{H}$ of $C_c(\mathcal{G})$
in the seminorm given by the inner product $\langle f|g\rangle_\phi=\phi(f^*\times g)$, the left multiplication operators
$L_f:C_c(\mathcal{G}) \ni g\longmapsto f\times g\in C_c(\mathcal{G})$ give rise to a $*$-representation
$\pi:C_c(\mathcal{G})\to\mathscr{B}(\mathscr{H})$, so by the definition of $\|\,.\,\|_{\text{\rm full}}$ we must have $\|L_f\|\leq\|f\|_{\text{full}}$.
\end{proof}
\begin{applemma}\label{app-ueta}
Given a unit $u\in\mathcal{G}^{(0)}$ and a positive definite function $\theta$ on 
the isotropy group $u\mathcal{G}u$,
the linear map $\eta_\theta:C_c(\mathcal{G})\ni f\longmapsto \sum_{\gamma\in u\mathcal{G}u}\theta(\gamma)f(\gamma)\in\mathbb{C}$
 is positive, in the sense of \eqref{app-phi-pos}. {\rm (For each $f\in C_c(\mathcal{G})$, the sum that defines 
$\eta_\theta(f)$ has only finitely many non-zero terms.)}
\end{applemma}

\begin{proof}
Let $\{\mathbf{v}_\gamma\}_{\gamma\in u\mathcal{G}u}$ be the standard unitary generators
of the group algebra $\mathbb{C}[u\mathcal{G}u]$ (so that
we can view $\mathbb{C}[u\mathcal{G}u]=\text{span}\{\mathbf{v}_\gamma\}_{\gamma\in
u\mathcal{G}u}$), and let $\omega$ be the positive linear linear map on $\mathbb{C}[u\mathcal{G}u]$ associated with
$\theta$, i.e.
$\omega(\sum_{\gamma\in F}\zeta_\gamma\mathbf{v}_\gamma)=\sum_{\gamma\in F}\theta(\gamma)\zeta_\gamma$
(for all finite sets $F\subset u\mathcal{G}u$).

Fix now some $f\in C_c(\mathcal{G})$, and let us justify the inequality
\begin{equation}
\eta_\theta(f^*\times f)\geq 0.
\label{app-phitheta-pos}
\end{equation}
Consider the finite set $X=\text{supp}\,f\cap\mathcal{G}u$, and assume $X\neq\varnothing$
(otherwise, $\eta_\theta(f^*\times f)=0$, and there is nothing to prove).
Let us list the finite set of units $r(X)=\{u_1,\dots,u_k\}$, and fix, for each $j$, an element $\gamma_j\in X$, such that $r(\gamma_j)=u_j$, along with an open bisection $\mathcal{B}_j\ni\gamma_j$, and some normalizer
$n_j\in C_c(\mathcal{B}_j)$ with $n_j(\gamma_j)=1$. By suitably shrinking these bisections, we can assume that the sets $r(\mathcal{B}_j)$, $j=1,\dots,k$ are disjoint.

Using
\eqref{nfn*}, it is easy to see that
$$(f^*\times f)(\gamma)=\sum_{j=1}^k(f^*\times n_j^*\times n_j\times f)(\gamma),\,\,\,\forall\,\gamma\in 
u\mathcal{G}u,$$
so it suffices to prove \eqref{app-phitheta-pos} with each of $n_j\times f$'s in place of $f$; in other words, it suffices to prove  \eqref{app-phitheta-pos} for a function $f\in C_c(\mathcal{G})$ which has
$\mathcal{G}u\cap \text{supp}\,f\subset u\mathcal{G}u$. However, this special case is trivial, since for any such $f$,
when we consider the element $\hat{f}=\sum_{\gamma\in \text{supp}\,f\cap u\mathcal{G}u} f(\gamma)\mathbf{v}_\gamma
\in\mathbb{C}[u\mathcal{G}u]$,
we have $\eta_\theta(f^*\times f)=\omega(\hat{f}^*\times \hat{f})\geq 0$.
\end{proof}

\begin{apptheorem}
Assume $\mathcal{Y}\subset\text{\rm Iso}(\mathcal{G})$ is an open subgroupoid. 
%
Then the $*$-algebra inclusion 
\begin{equation}
C_c(\mathcal{Y})\subset C_c(\mathcal{G})
\label{appA-thm}
\end{equation}
 is $\|\,.\,\|_{\text{\rm full}}$-isometric, thus it
extends to a $C^*$-inclusion $C^*(\mathcal{Y})\subset C^*(\mathcal{G})$.
\end{apptheorem}
\begin{proof}
In order to distinguish between the full norms on $C^*(\mathcal{Y})$ and $C^*(\mathcal{G})$, we will denote them by 
$\|\,.\,\|_{C^*(\mathcal{Y})}$ and $\|\,.\,\|_{C^*(\mathcal{G})}$, respectively.

As pointed out in Section \ref{bundles-sec}, for each unit $u\in\mathcal{Y}^{(0)}$,
we have a surjective $*$-homomorphism
$\rho_u:C^*(\mathcal{Y})\to C^*(\mathcal{Y}u)$ (from the full groupoid $C^*$-algebra to the full group $C^*$-algebra),
which when restricted to $C_c(\mathcal{Y})$, acts as:
$$\rho_u:C_c(\mathcal{Y})\ni f\longmapsto f\big|_{\mathcal{Y}u}\in \mathbb{C}[\mathcal{Y}u]).$$
Using \eqref{norm-sup-punif}, we also know that
\begin{equation}
\|a\|_{C^*(\mathcal{Y})}=
\sup_{u\in\mathcal{Y}^{(0)}}
\|\rho_u(a)\|_{C^*(\mathcal{Y}u)},\,\,\,\forall\,a\in C^*(\mathcal{Y})
\label{app-sup-u}
\end{equation}
where $\|\,.\,\|_{C^*(\mathcal{Y}u)}$ denotes the norm in the full group $C^*$-algebra of $\mathcal{Y}u$.

Consider now the set
$$\Xi=\{(u,\theta)\,:\,u\in \mathcal{Y}^{(0)},\,\theta\text{ positive definite function on the group }
\mathcal{Y}u\}.$$
Fix for the moment $(u,\theta)\in\Xi$.
Let $\varepsilon_\theta:C^*(\mathcal{Y}u)\to\mathbb{C}$ be the positive linear functional associated to
$\theta$, and let $\psi_{(u,\theta)}=\varepsilon_\theta\circ\rho_u$. 
As $\theta$ is a positive definite function on $\mathcal{Y}u$, which is a subgroup in $u\mathcal{G}u$, extend it
(as zero outside $\mathcal{Y}u$) to a a positive definite function $\theta'$ on $u\mathcal{G}u$, then use
Lemma \ref{app-ueta} to produce a linear positive map
$$\eta_{\theta'}:C_c(\mathcal{G})\ni f
\longmapsto \sum_{\gamma\in u\mathcal{G}u}\theta'(\gamma)f(\gamma)
=\sum_{\gamma\in \mathcal{Y}u}\theta(\gamma)f(\gamma)
\in\mathbb{C}.$$
Since $\eta_{\theta'}$ acts on
$C_c(\mathcal{G}^{(0)})$ as a multiple ($\theta(e)$ or zero) of the evaluation map
$f\longmapsto f(u)$,
by Lemma \ref{app-positivity}, $\eta_{\theta'}$ extends to a positive linear functional, hereafter denoted by
$\phi_{(u,\theta)}$ on $C^*(\mathcal{G})$.

Now, when we consider the canonical $*$-homomorphism
$\pi:C^*(\mathcal{Y})\to C^*(\mathcal{G})$ arising from
\eqref{appA-thm}, we have a family $\Sigma=\{\phi_{(u \theta)}\}_{(u,\theta)\in\Xi}$ of positive linear maps on $C^*(\mathcal{G})$, for which the 
associated family 
\begin{align*}\Sigma^\pi&=\{\phi_{(u,\theta)}\circ \pi\}_{(u,\theta)\in\Xi}=
\{\psi_{(u,\theta)}\}_{(u,\theta)\in\Xi}=\\
&=\{\psi\circ\rho_u\,:\,u\in\mathcal{Y}^{(0)},\,\psi\text{ linear positive on }C^*(\mathcal{Y}u)\}
\end{align*}
is clearly jointly faithful on $C^*(\mathcal{Y})$, by
 \eqref{app-sup-u}, so the desired injectivity of $\pi$ follows from
Remark \ref{embed-rem}.
\end{proof}

\bibliographystyle{plain}

\end{document}